\documentclass[11pt]{article}
\usepackage[letterpaper,hmargin=1in,vmargin=1in]{geometry}
\usepackage{graphicx}                   
\usepackage{amsmath}                    
\usepackage{amssymb}                    
\usepackage{amsfonts}                   
\usepackage{url}                        
\usepackage{color}                      
\usepackage{cite}                       
\usepackage{hyperref}                   
\hypersetup{colorlinks=true,linkcolor=[rgb]{0,0,0},citecolor=[rgb]{0,0,0},urlcolor=[rgb]{0,0,0}}
\usepackage{amsthm}                     
\usepackage{thmtools}                    
\usepackage{mathrsfs}                   
\usepackage{xparse}
\usepackage{bbm}
\declaretheorem[name=Lemma,numberwithin=section]{lemma}

\newcommand{\ceil}[1]{\left\lceil #1\right\rceil}
\newcommand{\ang}[1]{\left\langle #1 \right\rangle}
\newcommand{\RE}{\mathbb{R}}            
\newcommand{\eps}{\varepsilon}          
\newcommand{\ST}{\,:\,}                 

\newcommand{\inv}[1]{\frac{1}{#1}}
\newcommand{\inner}[2]{\ang{#1,#2}} 

\newcommand{\polar}[1]{#1^{\circ}}
\newcommand{\bd}{\partial}
\newcommand{\etal}{\textit{et al.}}
\newcommand{\SP}{\kern+1pt}

\newcommand{\RLocal}{r^+}
\newcommand{\CLEQ}{\lesssim_{\kern+1pt d}}
\newcommand{\CGEQ}{\gtrsim_{\kern+1pt d}}
\newcommand{\CEQ}{\asymp_d}
\newcommand{\ccs}{c_{\SP\mathrm{c}}}
\newcommand{\ce}{c_{\SP\mathrm{e}}}
\newcommand{\cb}{c_{\SP\mathrm{b}}}

\DeclareMathOperator{\centroid}{centroid}
\DeclareMathOperator{\interior}{int}
\DeclareMathOperator{\cofactor}{cofactor} 

\DeclareMathOperator{\conv}{conv}

\DeclareMathOperator{\vol}{vol}
\DeclareMathOperator{\area}{area}
\DeclareMathOperator{\dist}{dist}
\DeclareMathOperator{\finsler}{Finsler}
\DeclareMathOperator{\expand}{expand}

\DeclareMathOperator{\Dom}{Dom}
\DeclareMathOperator{\Var}{Var}

\NewDocumentCommand{\volBus}{O{}}{\vol^{\mathrm{Bus}}_{#1}}     
\NewDocumentCommand{\areaBus}{O{}}{\area^{\mathrm{Bus}}_{#1}}   
\NewDocumentCommand{\volMink}{O{}}{\vol^{M}_{#1}} 
\NewDocumentCommand{\volHilb}{O{}}{\vol^{H}_{#1}} 
\NewDocumentCommand{\volFunk}{O{}}{\vol^{F}_{#1}} 
\NewDocumentCommand{\volX}{O{}}{\vol_{#1}} 
\NewDocumentCommand{\areaMink}{O{}}{\area^{M}_{#1}} 
\NewDocumentCommand{\areaHilb}{O{}}{\area^{H}_{#1}} 
\NewDocumentCommand{\areaFunk}{O{}}{\area^{F}_{#1}} 
\NewDocumentCommand{\areaX}{O{}}{\area_{#1}} 
\NewDocumentCommand{\distMink}{O{}}{\dist^{M}_{#1}} 
\NewDocumentCommand{\distHilb}{O{}}{\dist^{H}_{#1}} 
\NewDocumentCommand{\distFunk}{O{}}{\dist^{F}_{#1}} 
\NewDocumentCommand{\distX}{O{}}{\dist_{#1}} 
\NewDocumentCommand{\finsHilb}{O{}}{\finsler^{H}_{#1}} 
\NewDocumentCommand{\finsFunk}{O{}}{\finsler^{F}_{#1}} 
\NewDocumentCommand{\finsX}{O{}}{\finsler_{#1}} 
\NewDocumentCommand{\polarIn}{O{}m}{{#2}^{\circ_{#1}}} 

\newcommand{\orthproj}[2]{{#1}\mathbin{\big|}{#2}} 

\NewDocumentCommand{\norm}{O{}}{n_{#1}} 
\NewDocumentCommand{\ballX}{O{}}{B_{#1}} 
\NewDocumentCommand{\ballMink}{O{}}{B^{M}_{#1}} 
\NewDocumentCommand{\ballHilb}{O{}}{B^{H}_{#1}} 
\NewDocumentCommand{\ballFunk}{O{}}{B^{F}_{#1}} 
\NewDocumentCommand{\finsBallX}{O{} d()}{B_{#1}(#2)}
\NewDocumentCommand{\finsBallHilb}{O{} d()}{B^{H}_{#1}(#2)}
\NewDocumentCommand{\finsBallFunk}{O{} d()}{B^{F}_{#1}(#2)}

\NewDocumentCommand{\Ncov}{O{} d()}{N^{H}_{#1}(#2)}
\NewDocumentCommand{\Nbd}{O{} d()}{S^{H}_{#1}(#2)}

\NewDocumentCommand{\NbdMink}{O{} d()}{S^{M}_{#1}(#2)}
\NewDocumentCommand{\NcovMink}{O{} d()}{N^{M}_{#1}(#2)}

\NewDocumentCommand{\NcovIn}{O{} d()}{\widehat{N}^{H}_{#1}(#2)}
\NewDocumentCommand{\NcovMinkIn}{O{} d()}{\widehat{N}^{M}_{#1}(#2)}
\NewDocumentCommand{\NbdIn}{O{} d()}{\widehat{S}^{H}_{#1}(#2)}
\NewDocumentCommand{\NbdMinkIn}{O{} d()}{\widehat{S}^{M}_{#1}(#2)}

\NewDocumentCommand{\expandM}{O{} d()}{\expand^{M}_{#1}(#2)}
\NewDocumentCommand{\expandH}{O{} d()}{\expand^{H}_{#1}(#2)}

\title{On the Duality of Coverings in Hilbert Geometry}

\author{%
	Sunil Arya\thanks{Research supported by the Research Grants Council of Hong Kong, China under project number 16214721.}\\
		Department of Computer Science and Engineering \\
		Hong Kong University of Science and Technology \\
		Clear Water Bay, Kowloon, Hong Kong\\
		arya@cse.ust.hk \\
		\and
	David M. Mount\\
		Department of Computer Science and \\
		Institute for Advanced Computer Studies \\
		University of Maryland \\
		College Park, Maryland 20742 \\
		mount@umd.edu \\
}

\begin{document}
\date{~}
\maketitle

\begin{abstract} 
We prove polarity duality for covering problems in Hilbert geometry. Let $G$ and $K$ be convex bodies in $\mathbb{R}^d$ where $G \subset \operatorname{int}(K)$ and $\operatorname{int}(G)$ contains the origin. Let $N^H_K(G,\alpha)$ and $S^H_K(G,\alpha)$ denote, respectively, the minimum numbers of radius-$\alpha$ Hilbert balls in the geometry induced by $K$ needed to cover $G$ and $\partial G$. Our main result is a Hilbert-geometric analog of the K\"{o}nig--Milman covering duality: there exists an absolute constant $c \geq 1$ such that for any $\alpha \in (0,1]$, 
\[
    c^{-d}\,N^H_{G^{\circ}}(K^{\circ},\alpha) 
        ~ \leq ~ N^H_K(G,\alpha) 
        ~ \leq ~ c^{d}\,N^H_{G^{\circ}}(K^{\circ},\alpha), 
\] 
and likewise, 
\[
    c^{-d}\,S^H_{G^{\circ}}(K^{\circ},\alpha) 
        ~ \leq ~ S^H_K(G,\alpha) 
        ~ \leq ~ c^{d}\,S^H_{G^{\circ}}(K^{\circ},\alpha). 
\] 
We also recover the classical volumetric duality for translative coverings of centered convex bodies, and obtain a new boundary-covering duality in that setting.

The Hilbert setting is subtler than the translative one because the metric is not translation invariant, and the local Finsler unit ball depends on the base point. The proof involves several ideas, including $\alpha$-expansions, a stability lemma that controls the interaction between polarity and expansion, and, in the boundary case, a localized relative isoperimetric argument combined with Holmes--Thompson area estimates. In addition, we provide an alternative proof of Faifman's polarity bounds for Holmes--Thompson volume and area in the Funk and Hilbert geometries. 
\end{abstract} 


\section{Introduction} \label{s:introduction}

Coverings of convex bodies are a basic quantitative concept in convex geometry. Coverings have found diverse applications, including approximations of low combinatorial complexity~\cite{AAFM22, AFM17c, AFM26}, approximate nearest neighbor searching~\cite{AFM18a}, computing the diameter and $\eps$-kernels~\cite{AFM17b}, and approximations to the closest vector problem~\cite{NaV22, AFM24}. They can be constructed from various classes of shapes, including ellipsoids and parallelepipeds~\cite{EHN11}, scaled Macbeath regions~\cite{AAFM22, AFM23, NaV22}, and Hilbert balls~\cite{AbM18, ArM23}. 

A central theme in the subject is that covering behavior is often reflected by polarity, in particular by the strong connection between the geometry of a convex body and its polar, or, equivalently, between a normed space and its dual. For example, K{\"o}nig and Milman~\cite{KoM87} showed that given two centrally symmetric convex bodies $C$ and $D$ in $\RE^d$ (or equivalently, two normed spaces in $\RE^d$), the covering complexity of $C$ by translates of $D$ and the covering complexity of the polar $\polar{D}$ by translates of $\polar{C}$ are equivalent up to a constant factor of the form $c^d$. This is related to the duality conjecture for entropy numbers of linear operators, posed by Pietsch~\cite{Pie72}. A major milestone is the duality theorem for entropy numbers by Artstein, Milman, and Szarek~\cite{AMS04}, together with related refinements comparing the entropy of an operator with that of its adjoint. (See also Artstein-Avidan, Giannopoulos, and Milman~\cite{AGM15} for a wider exploration of this area.)

In this paper, we establish a polarity principle for intrinsic covering problems in Hilbert geometries. Hilbert geometry provides a classical geometric framework in which the metric structure and associated notions of volume and area are tied to an ambient convex domain $K$. (Definitions will be provided in Section~\ref{s:prelim}.) Its basic features are that line segments are geodesics and that the metric is invariant under projective transformations. For each $x \in \interior(K)$ and $\alpha > 0$, let $\ballHilb[K](x, \alpha)$ denote the Hilbert metric ball of radius $\alpha$ centered at $x$. Given convex bodies $G$ and $K$ in $\RE^d$, with $G \subset \interior(K)$, let $\Ncov[K](G,\alpha)$ denote the minimum number of Hilbert balls of radius $\alpha$ needed to cover $G$. Define $\Nbd[K](G,\alpha)$ analogously for covering $G$'s boundary. Our main result is a duality property for these two covering numbers, which holds throughout the small-radius regime $\alpha \in (0,1]$.

\begin{restatable}[Duality for Hilbert-ball coverings]{theorem}{thmHilbertCoverDuality} \label{thm:hilbert-cover-duality}
There exists an absolute constant $c \geq 1$ such that for any pair of convex bodies $G$ and $K$ in $\RE^d$, with $O \in \interior(G)$ and $G \subset \interior(K)$, and any $\alpha \in (0,1]$,
\begin{gather*}
    (i)~~ c^{-d}\,\Ncov[\polar{G}](\polar{K},\alpha) 
        ~ \leq ~ \Ncov[K](G,\alpha) 
        ~ \leq ~ c^{d}\,\Ncov[\polar{G}](\polar{K},\alpha),
    \\
    \text{and}
    \\
    (ii)~~ c^{-d}\,\Nbd[\polar{G}](\polar{K},\alpha) 
        ~\leq ~ \Nbd[K](G,\alpha) 
        ~ \leq ~ c^{d}\,\Nbd[\polar{G}](\polar{K},\alpha).
\end{gather*}
\end{restatable}

Covering by Hilbert balls differs significantly from covering by translates of a convex body, which (at least in the centrally symmetric case) can be couched as covering by metric balls in a Minkowski norm. In the Hilbert geometry, both the metric and the intrinsic volume and area densities depend on the ambient domain, and polarity must be reconciled with a geometry that varies from point to point. In the limit, as $\alpha \to 0$, volumetric and boundary covering numbers are closely related to a body's volume and surface area. From this perspective, Faifman's recent results on the duality properties of the Holmes--Thompson volume and surface areas of convex bodies in Hilbert geometries anticipate our work~\cite{Fai24}.

To the best of our knowledge, the exploration of duality properties for boundary coverings is novel. Boundary coverings are of particular interest in the design of approximation algorithms involving convex bodies, including polytope membership and nearest-neighbor searching~\cite{AFM18a}. Extending the work on volumetric covers of convex bodies by K{\"o}nig and Milman~\cite{KoM87}, Milman and Pajor~\cite{MiP00}, and Artstein {\etal}~\cite{AMS04}, we present a duality result for covering the boundary of a convex body by convex translates. For convex bodies $C$ and $D$ in $\RE^d$ and $\alpha>0$, let $\NbdMink[D](C,\alpha)$ denote the minimum number of translates of $\alpha D$ needed to cover $\bd C$. Our result applies under the assumption that the bodies are \emph{centered}, meaning that the body's centroid coincides with the origin.

\begin{restatable}[Duality for boundary coverings by translates]{theorem}{thmMinkAsymBoundaryCoverDuality} \label{thm:mink-asym-boundary-cover-duality}
There exists an absolute constant $c \geq 1$ such that for any pair of centered convex bodies $C$ and $D$ in $\RE^d$ and any $\alpha>0$,
\[
    c^{-d}\,\NbdMink[\polar{C}](\polar{D},\alpha) 
        ~ \leq ~ \NbdMink[D](C,\alpha)
        ~ \leq ~ c^{d}\,\NbdMink[\polar{C}](\polar{D},\alpha).
\]
\end{restatable}

Taken together, these results reveal a common duality mechanism across two settings. Thus, Hilbert covering numbers are not merely an analog of the classical translative theory, but part of a broader polarity framework for coverings. 

\subsection{Overview of Techniques} \label{s:techniques-overview}

For volumetric coverings, we first thicken the inner body $G$ by including all points within Hilbert distance $\alpha$ of $G$, which we call its \emph{$\alpha$-expansion}, denoted $G_+$ (see Figure~\ref{f:overview}(a)). A maximal $\alpha$-separated subset of $G$ then yields a packing by radius-$\frac{\alpha}{2}$ Hilbert balls inside $G_+$ (see Figure~\ref{f:overview}(b)), and bounded-radius volume growth estimates convert this packing into a bound in terms of the normalized Holmes--Thompson volume of $G_+$. We next dualize this volume bound using Faifman's Hilbert polarity estimates~\cite{Fai24}. The resulting dual volume estimate is naturally expressed in the ambient geometry induced by $\polar{G_+}$, whereas the target covering problem is formulated in the ambient geometry induced by $\polar{G}$ (see Figure~\ref{f:overview}(c)). Since $G \subseteq G_+$, polarity reverses inclusion and hence $\polar{G_+} \subseteq \polar{G}$, so on the polar side the ambient body has shrunk, while the set to be covered remains the same, namely $\polar{K}$.

\begin{figure}[htbp]
  \centerline{\includegraphics[scale=0.40,page=1]{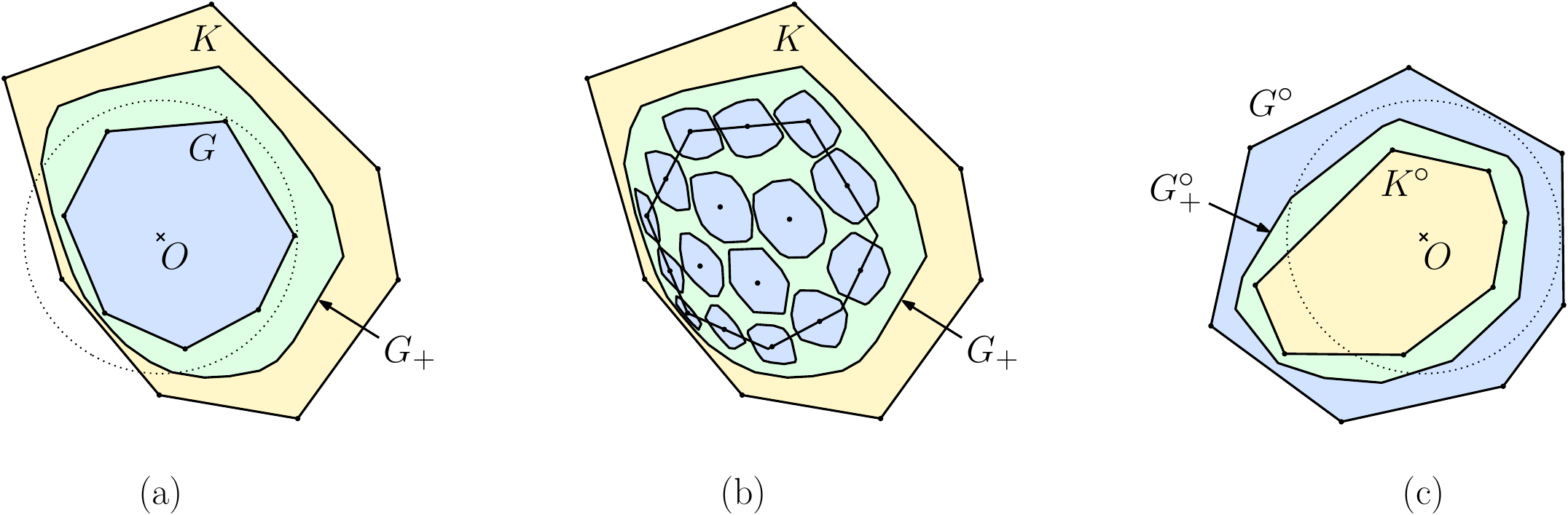}}
  \caption{Overview of techniques.}
  \label{f:overview}
\end{figure}

The issue is therefore to control how much the Hilbert metric on this fixed set changes when the ambient geometry passes from that of $\polar{G}$ to that of $\polar{G_+}$. In Section~\ref{s:distortion}, we prove a polarity-expansion stability lemma, which shows that on $\polar{K}$ this change is limited, in the sense that distances increase by at most an additive $2\alpha$, so every radius-$\alpha$ ball in the geometry of $\polar{G}$ induces a radius-$3\alpha$ ball in the geometry of $\polar{G_+}$ (see Lemma~\ref{lem:polar-exp-hilb}). This allows the two covering problems to be compared up to a universal constant-factor change in radius, and hence up to the acceptable $c^d$ loss in covering number.

For boundary coverings, the additional difficulty is that a radius-$\alpha$ Hilbert ball centered on $\bd G$ may capture arbitrarily little boundary area. This issue is most acute with thin bodies. Here, the expansion plays a genuine regularizing role by ``fattening'' the body. In Section~\ref{s:prop-exp-surf}, we show that $G_+$ is relatively fat at scale $\alpha$, in the sense that every radius-$\alpha$ Hilbert ball centered on $\bd G_+$ contains a dimension-dependent fraction of its volume inside $G_+$. We then combine this with a localized relative isoperimetric inequality, which intuitively provides a lower bound on the boundary-area fraction cut off inside the ball in terms of the corresponding captured volume fraction. This yields the boundary-covering estimates required to apply the same duality scheme as in the volumetric case.

We apply the same expansion--duality strategy for covering by translates. We first consider the case of centrally symmetric bodies, viewing them from the perspective of covering by Minkowski balls. One again passes to an $\alpha$-expansion of the body $C$ to be covered, derives the relevant volumetric or boundary estimate for this expansion, and dualizes it. As in the Hilbert case, the dualized estimate is naturally expressed in the geometry induced by the expanded polar body $\polar{C_+}$ rather than the target geometry induced by $\polar{C}$, and the key step is therefore a polarity-expansion stability principle, which transfers the associated covering problem on $\polar{D}$ back to the correct ambient geometry.

The same framework is flexible enough to extend beyond the centrally symmetric setting. In the spirit of classical symmetrization arguments, we use two centrally symmetric derivatives of a convex body, the symmetric core and symmetric union (defined in Section~\ref{s:prelim}). These reductions pass the duality argument through these symmetric intermediaries, where the area and covering estimates can be determined. We then return the result to the original bodies. Combined with the expansion-based regularization and polarity-expansion stability principles developed above, this yields corresponding extensions to centered convex bodies in both the volumetric and boundary settings.

\subsection{Organization}

The remainder of the paper is organized as follows. Section~\ref{s:prelim} introduces notation and terminology, which will be used throughout the paper. Section~\ref{s:toolbox} develops tools for analyzing the geometry of regions of bounded radius. This introduces a mechanism based on Macbeath regions for estimating local volumes and areas. Next, Section~\ref{s:distortion} studies the interaction between polarity and expansion and proves the stability lemmas that control the change in ambient geometry after dualization. Section~\ref{s:prop-exp-vol} establishes the expansion properties needed for volumetric coverings. 

Given these tools, we next present our volumetric duality results. In Section~\ref{s:volume-cover} we present (known) volumetric duality of covering numbers for the centrally symmetric Minkowski case as a warm-up, and we then extend this to the Hilbert case, establishing the volumetric part of Theorem~\ref{thm:hilbert-cover-duality}. 

We then turn to boundary coverings. Section~\ref{s:rel-iso} states and proves localized relative isoperimetric inequalities. Section~\ref{s:prop-exp-surf} develops the additional expansion properties needed at the boundary. Finally, Section~\ref{s:boundary-cover} proves the boundary duality results for both the Hilbert case (Theorem~\ref{thm:hilbert-cover-duality}(ii)) and the case of covering by translates of centered bodies (Theorem~\ref{thm:mink-asym-boundary-cover-duality}).

Finally, proofs of selected technical results, including the Funk Holmes--Thompson polarity identities, are deferred to the appendix.

\section{Preliminaries} \label{s:prelim}

We present the notation and terminology used throughout the paper. Unless otherwise stated, we assume $d \geq 2$. We use $\inner{\cdot}{\cdot}$ to denote the standard inner product on $\RE^d$, and $\|\cdot\| = \sqrt{\inner{\cdot}{\cdot}}$ for the Euclidean norm. Let $O$ denote the origin, let $B_2^d$ denote the Euclidean unit ball centered at the origin, and let $S^{d-1}$ denote its boundary, the unit sphere. Given a linear subspace $E \subset \RE^d$ and any set $K \subset \RE^d$, let $\orthproj{K}{E}$ denote the orthogonal projection of $K$ onto $E$.

A \emph{convex body} in $\RE^d$ is a compact convex set with nonempty interior. Given a convex set $K$, we denote its boundary and interior by $\bd K$ and $\interior(K)$, respectively. For $\alpha \geq 0$, $\alpha K$ denotes uniform scaling of $K$ about the origin, and for $x \in \RE^{d}$, $K + x$ denotes translation of $K$ by $x$. For $u \in \RE^d$, the \emph{support function} of $K$ is \[
h_K(u) = \sup \{ \inner{u}{x} : x \in K \}.
\]

Given any nonempty set $K \subset \RE^d$, its \emph{polar}, denoted $\polar{K}$, is defined by
\[
	\polar{K}
		~ = ~ \{ y \in \RE^d \ST \inner{x}{y} \leq 1, \text{~for all $x \in K$} \}.
\]
The polar is always closed, convex, and contains the origin (see, e.g., Barvinok~\cite{Bar02}). If $K$ is a convex body with $O \in \interior(K)$, then $\polar{K}$ is a convex body with $O \in \interior(\polar{K})$ and $\polar{(\polar{K})} = K$. Moreover, polarity reverses inclusion; if $K_1 \subseteq K_2$, then $\polar{K_2} \subseteq \polar{K_1}$.

It will be useful to define a subspace restriction of the polar. For a linear subspace $E \subseteq \RE^d$ and a nonempty set $G\subseteq E$, we define the polar of $G$ in $E$ by
\[
    \polarIn[E]{G}
        ~ := ~ \{y\in E : \inner{x}{y} \leq 1, \text{~for all $x\in G$}\}.
\]
The following lemma states the well-known duality between slicing a convex body and projecting its polar (see \cite[Theorem 2.2.9 and Corollary 2.2.10]{Tho96}).

\begin{lemma}[Projection-Section Duality] \label{lem:slice}
Let $K \subseteq \RE^d$ be a convex body with $O \in \interior(K)$ and $E \subseteq \RE^d$ be a linear subspace. Then 
\[
    \polarIn[E]{(K \cap E)} 
        ~ = ~ \orthproj{\polar{K}}{E}.
\]
\end{lemma}

\subsection{Measures and Smoothness} \label{s:measures}

For $1 \leq k \leq d$, let $\lambda_k$ denote the $k$-dimensional Hausdorff measure on $\RE^d$. In particular, $\lambda_d$ coincides with the $d$-dimensional Lebesgue measure, and on any $k$-dimensional $C^1$ submanifold of $\RE^d$, $\lambda_k$ agrees with the standard $k$-dimensional surface area. Let $\omega_d = \lambda_d(B_2^d)$.

For a convex body $K$, its \emph{centroid} is defined in the usual manner as
\[
    \centroid(K) ~ := ~ \frac{1}{\lambda_d(K)}\int_K x\, d\lambda_d(x). 
\]
We say that $K$ is \emph{centered} if $\centroid(K) = O$. We say that $K$ is \emph{centrally symmetric about $z\in\RE^d$} if $K-z = -(K-z)$, and we say that $K$ is \emph{centrally symmetric} if it is centrally symmetric about $O$. 

Given a convex body $K$, a point $p \in \bd K$ is \emph{smooth} if $\bd K$ has a unique supporting hyperplane at $p$, or equivalently, a unique outer unit normal. For a smooth point $x \in \bd K$, let $\norm[K](x)$ denote the outer unit normal at $x$. It is well known that the set of nonsmooth points of $\bd K$ has $(d-1)$-dimensional Hausdorff measure zero (see \cite[Theorem~2.2.5]{Sch14}). Accordingly, whenever we integrate over $\bd K$ an expression depending on the outer unit normal or on the supporting hyperplane, it is understood to be evaluated at smooth points, which form a set of full $\lambda_{d-1}$-measure in $\bd K$. In particular, any inequality for a surface functional defined by such an integral and proved for piecewise $C^1$ surfaces applies verbatim to $\bd K$ for an arbitrary convex body $K$.

\subsection{Funk and Hilbert Geometries} \label{s:funk-hilbert}

Let $K\subset\RE^d$ be a convex body with nonempty interior. For distinct points $x, y \in \interior(K)$, let $y'$ denote the point where the ray from $x$ through $y$ intersects $\bd K$ (see Figure~\ref{f:funk-hilbert}(a)). The \emph{Funk distance} with respect to $K$, denoted $\distFunk[K](\cdot,\cdot)$, is 
\[
    \distFunk[K](x,y)
       ~ := ~ \ln \frac{\|x - y'\|}{\|y - y'\|},
\]
with $\distFunk[K](x,x) = 0$. The Funk distance is nonnegative, asymmetric, satisfies the triangle inequality, and is invariant under invertible affine transformations~\cite{PaT14}. Because of its asymmetry, this is often called a weak metric or a quasi-metric.

\begin{figure}[htbp]
  \centerline{\includegraphics[scale=0.40,page=1]{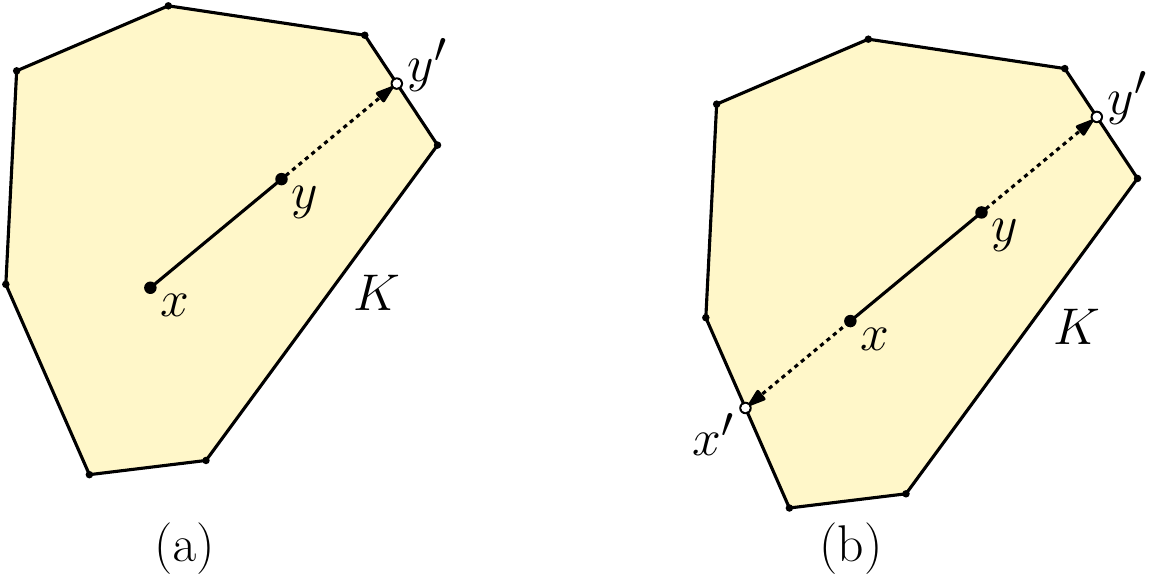}}
  \caption{(a) The Funk distance and (b) the Hilbert distance.}
  \label{f:funk-hilbert}
\end{figure}

The \emph{Hilbert distance} with respect to $K$, denoted $\distHilb[K](\cdot,\cdot)$, symmetrizes the Funk distance. Letting $x'$ be the point where the ray from $y$ through $x$ intersects $\bd K$,
\[
    \distHilb[K](x,y) 
        ~ := ~ \inv{2} \left( \distFunk[K](x,y) + \distFunk[K](y,x) \right) 
        ~ =  ~ \inv{2} \ln \frac{\|y-x'\|}{\|x-x'\|} \, \frac{\|x-y'\|}{\|y-y'\|}
\]
(see Figure~\ref{f:funk-hilbert}(b)). This defines a metric that is invariant under invertible projective transformations, since the logarithmic term is the cross ratio $(x, y; y', x')$. It is additive on collinear triples, that is, if $x,y,z$ are collinear and $y$ lies between $x$ and $z$, then $\distHilb[K](x,y) + \distHilb[K](y,z) = \distHilb[K](x,z)$. Given $x \in \interior(K)$ and nonnegative $r$, let $\ballHilb[K](x,r)$ denote the Hilbert metric ball of radius $r$ centered at $x$. It is well known that Hilbert balls are convex bodies~\cite[Section~18]{Bus55} (see also~\cite[Corollary~4.6]{Tro14}).

We next introduce the associated Finsler structures. Abstractly, a \emph{Finsler metric} on a manifold is a continuous function $F$ on the tangent bundle whose restriction to each tangent space $T_x$ is a (possibly asymmetric) norm. For each point $x$, the associated (closed) \emph{Finsler ball} in $T_x$ is the set
\[
     \{v \in T_x \ST F(x,v) \leq 1\}.
\]

Perhaps the most basic example is the Minkowski functional. Given a convex body $D$ in $\RE^d$ with $O \in \interior(D)$ and $u \in \RE^d$, the \emph{Minkowski functional} (or \emph{gauge}) induced by $D$ is defined by
\[
    \|u\|_D ~ := ~ \inf \{\lambda > 0 \ST u \in \lambda D\}.
\]
This functional is positively homogeneous and subadditive. It is a norm if and only if $D$ is centrally symmetric, and otherwise it is generally asymmetric. We may therefore view $\RE^d$ as a Finsler manifold by identifying each tangent space $T_x$ with $\RE^d$ and equipping it with the constant gauge $\|\cdot\|_D$. Under this identification, the Finsler unit ball at $x$ equals $D$, and the induced distance function is
\[
    \distMink[D](x,y) ~ := ~ \|y-x\|_D.
\]
The corresponding balls are scaled translates of $D$:
\[
    \ballMink[D](x,r) ~ = ~ x + rD, \qquad\text{for $r > 0$}.
\]
When $D$ is centrally symmetric, this is the usual Minkowski geometry with unit ball $D$. 

It is well known that both the Funk and Hilbert geometries induce Finsler structures~\cite{Tro14}. 
For the Funk Finsler structure, consider $x \in \interior(K)$ and identify $T_x$ with $\RE^d$. 
For $v \in T_x$, define
\[
    \finsFunk[K](x,v) ~ := ~ \|v\|_{K-x}.
\]
The unit ball, called the \emph{Funk Finsler ball} $\ballFunk[K](x)$, is just $K - x$ (see Figure~\ref{f:finsler}(a)). Thus, the local geometry at $x$ is the gauge geometry determined by $K-x$. 

\begin{figure}[htbp]
  \centerline{\includegraphics[scale=0.40]{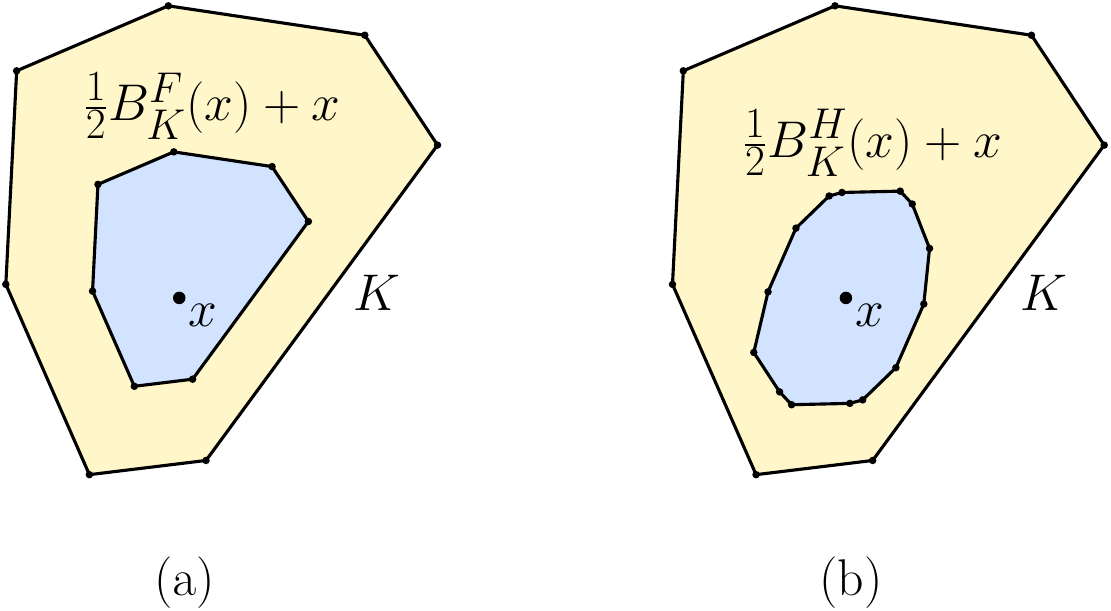}}
  \caption{The $\frac{1}{2}$-scaled Finsler balls in the (a) Funk and (b) Hilbert geometries (recentered on $x$).}
  \label{f:finsler}
\end{figure}

The Hilbert Finsler structure symmetrizes the Funk structure by taking the arithmetic mean. For $x \in \interior(K)$ and $v \in T_x$, define
\[
    \finsHilb[K](x,v) ~ := ~ \frac{1}{2} \big( \|v\|_{K-x} + \|{-v}\|_{K-x} \big).
\]
The resulting unit ball, $\ballHilb[K](x)$, called the \emph{Hilbert Finsler ball}, is centrally symmetric (see Figure~\ref{f:finsler}(b)).

\subsection{Covering numbers} \label{s:covering} 

Throughout the paper, we fix a localization radius $\RLocal \geq 8$. Whenever we invoke a local estimate in Hilbert geometry, the relevant radii are understood to be at most $\RLocal$, unless stated otherwise. The implicit constants in such estimates may depend on $\RLocal$, but are independent of the dimension and of the convex bodies involved. 

Since many of our estimates are sharp only up to factors that grow exponentially with the dimension, we adopt the following shorthand. We write $X \CLEQ Y$ if there exists a constant $c \geq 1$, independent of $d$, such that $X \leq c^{\SP d} Y$. We define $X \CGEQ Y$ by $Y \CLEQ X$, and write $X \CEQ Y$ if both inequalities hold. Any additional dependence of the implicit constant is indicated explicitly. 

Let $K$ be a convex body in $\RE^d$, let $\alpha > 0$, and let $U \subseteq \interior(K)$. We define the \emph{Hilbert covering number} of $U$ at scale $\alpha$ by 
\[
    \Ncov[K](U,\alpha)
        ~ := ~ \min \left\{ m \ST \exists\, x_1,\dots,x_m \in \interior(K),~ U \subseteq \bigcup_{i=1}^m \ballHilb[K](x_i, \alpha) \right\}.
\]
Given a convex body $G \subseteq \interior(K)$, define its \emph{Hilbert boundary covering number}, denoted $\Nbd[K](G,\alpha)$, to be $\Ncov[K](\bd G, \alpha)$.

We define the corresponding \emph{translative covering numbers} analogously. Given any convex body $D$ in $\RE^d$ with $O \in \interior(D)$, let $\alpha>0$, and let $U \subseteq \RE^d$. We define
\[
    \NcovMink[D](U,\alpha)
        ~ := ~ \min \left\{ m \ST \exists\, x_1,\dots,x_m \in \RE^d,~
        U \subseteq \bigcup_{i=1}^m (x_i + \alpha D) \right\}.
\]
Given a convex body $C$ in $\RE^d$, define
\[
    \NbdMink[D](C,\alpha) ~ := ~ \NcovMink[D](\bd C,\alpha).
\]
Observe that when $D$ is centrally symmetric, the covering bodies $x_i + \alpha D$ are $\alpha$-balls in the Minkowski geometry defined by $D$, and these coincide with the standard covering numbers in the Minkowski geometry. 

The following estimates will be used repeatedly. The translative covering estimate is elementary and holds at all scales, whereas the Hilbert estimate for small radii follows from the localization results proved later in Section~\ref{s:toolbox}. The proofs are deferred to Appendix~\ref{app:covering-lemmas}.

\begin{restatable}[Minkowski covering bound]{lemma}{lemMinkBallCover} \label{lem:mink-ball-cover} 
Let $D$ be a centrally symmetric convex body in $\RE^d$, and let $0 < r \leq r'$. Then, for any $U \subseteq \RE^d$, 
\[
    \NcovMink[D](U,r)
        ~ \CLEQ ~ \left( \frac{r'}{r} \right)^d \NcovMink[D](U,r'). 
\]
\end{restatable}

\begin{restatable}[Hilbert covering bound]{lemma}{lemHilbBallCover} \label{lem:hilb-ball-cover} 
Let $K$ be a convex body in $\RE^d$, and let $0 < r \leq r' \leq \frac{\RLocal}{2}$. Then, for any $U \subseteq \interior(K)$, 
\[
    \Ncov[K](U,r)
        ~ \CLEQ ~ \left( \frac{r'}{r} \right)^d \Ncov[K](U,r'). 
\]
\end{restatable}

\subsection{Holmes--Thompson Measures} \label{s:ht-vol-area}

Several definitions of volume and surface area exist in Finsler spaces. They involve scaling the Lebesgue measure at each point according to the local geometry as expressed through the Finsler ball. In the Holmes--Thompson approach~\cite{HoT79}, the scale factor is proportional to the Lebesgue volume of the polar of the Finsler ball.

Consider first the Minkowski geometry defined by a full-dimensional, centrally symmetric convex body $D$ in $\RE^d$. Let $\|\cdot\|_D$ denote the associated norm whose unit ball is $D$. The \emph{Minkowski Holmes--Thompson volume element} at $x \in \RE^d$ is defined to be
\[
    d \volMink[D](x) 
        ~ := ~ \frac{1}{\omega_d} \, \lambda_d(\polar{D}) \, d \lambda_d(x),
\]
where $d\lambda_d(x)$ denotes the Lebesgue volume element and $\omega_d$ is the Lebesgue volume of a $d$-dimensional Euclidean unit ball. The \emph{Minkowski Holmes--Thompson volume} of $U \subset \RE^d$ is
\[
    \volMink[D](U) 
        ~ = ~ \int_{x \in U} d \volMink[D](x) 
        ~ = ~ \frac{\lambda_d(\polar{D})}{\omega_d} \, \lambda_d(U).
\]

For a smooth $(d-1)$-dimensional surface $S$ and a point $x \in S$, we identify the tangent space $T_x$ with the corresponding $(d-1)$-dimensional subspace of $\RE^d$, equipped with Lebesgue measure $\lambda_{d-1}$. Recalling the restricted polar notation introduced just prior to Lemma~\ref{lem:slice}, the \emph{Minkowski Holmes--Thompson $(d-1)$-area element} is defined to be
\[
    d \areaMink[D](x) 
        ~ := ~ \frac{1}{\omega_{d-1}} \lambda_{d-1} \big( \polarIn[T_x]{(D \cap T_x)} \big) \, d \lambda_{d-1}(x).
\]
By Lemma~\ref{lem:slice}, this can be expressed equivalently as
\[
    d \areaMink[D](x) 
        ~ = ~ \frac{1}{\omega_{d-1}} \lambda_{d-1}\left(\orthproj{\polar{D}}{T_x} \right) \, d \lambda_{d-1}(x).
\]
The \emph{Minkowski Holmes--Thompson area} of $S$ is
\[
    \areaMink[D](S) ~ = ~ \int_{x \in S} d \areaMink[D](x).
\]

For the Hilbert geometry at $x \in \interior(K)$, let $B_x = \ballHilb[K](x)$ denote the Hilbert Finsler ball at $x$. Since $B_x$ is centrally symmetric, the \emph{Hilbert Holmes--Thompson volume} and \emph{surface-area elements} are defined by
\[
    d \volHilb[K](x) ~ := ~ d \volMink[B_x](x)
    \quad\text{and}\quad
    d \areaHilb[K](x) ~ := ~ d \areaMink[B_x](x),
\]
respectively. For the Funk geometry, the Holmes--Thompson volume and area elements are defined by the same local formulas, with $B_x = \ballFunk[K](x)$.

To simplify notation, we adopt the convention that the surface area of a convex body means the surface area of its boundary. That is, when $G$ is a convex body in $\RE^d$, we use $\areaMink[D](G)$, $\areaFunk[K](G)$, and $\areaHilb[K](G)$ to denote the Holmes--Thompson $(d-1)$-areas of $\bd G$. Similarly, when $S$ is a piecewise $C^1$ surface, these denote the intrinsic $(d-1)$-area of $S$. Unless otherwise stated, all references to volumes and surface areas are understood to be in the Holmes--Thompson sense.

A measure $\mu$ is \emph{monotonic under inclusion} if for any convex bodies $G,G'$ where $G \subseteq G'$, $\mu(G)\leq \mu(G')$. A useful feature of Holmes--Thompson surface area is its monotonicity under inclusion.  

\begin{lemma} \label{lem:HT-area-monotone}  
The Holmes--Thompson volumes and surface areas for both Minkowski geometries and Hilbert geometries are monotone under inclusion.
\end{lemma}

\begin{proof}
The monotonicity of volumes is trivially true. To see that the surface area measures are monotonic, observe that in both settings, Holmes--Thompson $(d-1)$-area admits a Crofton representation with a positive Crofton measure. For the Minkowski case, see Schneider~\cite[Theorem~5.1]{Sch05}. For the Hilbert case, see Vernicos--Walsh~\cite[Lemma~7]{VeW21} and Schneider~\cite[Theorem~2]{Sch01b}. The conclusion, therefore, follows from the standard Crofton argument, as in the proof of \cite[Lemma~7]{VeW21}. 
\end{proof}

\subsection{Duality of Holmes--Thompson Measures} \label{s:measure-inputs}

In this section, we present the measure-theoretic identities that will be used later. Proofs are omitted here, since these results are available in the literature. The following lemma is due to Holmes and Thompson \cite{HoT79}.

\begin{lemma}[Duality of Minkowski Holmes--Thompson measures] \label{lem:HT-mink-polarity}
Let $C$ and $D$ be centrally symmetric convex bodies in $\RE^d$. Then:
\begin{enumerate}
\item[$(i)$] $\volMink[D](C) = \volMink[\polar{C}](\polar{D})$ and
\item[$(ii)$] $\areaMink[D](C) = \areaMink[\polar{C}](\polar{D})$.
\end{enumerate}
\end{lemma}

Faifman established analogous results for the Funk geometry \cite{Fai24}. For completeness, we include proofs in Appendix~\ref{app:funk-polarity}. The volume proof (part~(i)) serves as a warm-up, and the area proof (part~(ii)) is elementary and avoids the symplectic machinery used in prior work.

\begin{restatable}[Duality of Funk Holmes--Thompson measures]{lemma}{funkHTDuality} \label{lem:funk-HT-polarity}
Let $G$ and $K$ be convex bodies in $\RE^d$ with $O \in \interior(G)$ and $G \subset \interior(K)$. Then:
\begin{enumerate}
\item[$(i)$] $\volFunk[K](G) = \volFunk[\polar{G}](\polar{K})$ and
\item[$(ii)$] $\areaFunk[K](G) = \areaFunk[\polar{G}](\polar{K})$.
\end{enumerate}

\end{restatable}

For Hilbert geometry, Faifman obtained the following polarity bounds for Holmes--Thompson volume and area by combining the above result with Hilbert--Funk comparison estimates. These bounds are the measure-theoretic input used in our Hilbert covering arguments.

\begin{lemma}[Duality of Hilbert Holmes--Thompson measures] \label{lem:HT-hilb-polarity}
Let $G$ and $K$ be convex bodies in $\RE^d$ with $O \in \interior(G)$ and $G \subset \interior(K)$. Let
$\beta_k := \binom{2k}{k}/2^k$ for $k \geq 1$. Then: 
\begin{enumerate}
\item[$(i)$] $\beta_d^{-1} \cdot \volHilb[\polar{G}](\polar{K}) ~ \leq ~ \volHilb[K](G) ~ \leq ~ \beta_d \cdot \volHilb[\polar{G}](\polar{K})$ and 
\item[$(ii)$] $\beta_{d-1}^{-1} \cdot \areaHilb[\polar{G}](\polar{K}) ~ \leq ~ \areaHilb[K](G) ~ \leq ~ \beta_{d-1} \cdot \areaHilb[\polar{G}](\polar{K})$. 
\end{enumerate}
\end{lemma}

\section{Localization and Comparison Tools}\label{s:toolbox}

A key concept we will use is that, locally, Hilbert geometry can be approximated by Minkowski geometry. In particular, within Hilbert balls of bounded radius, the Hilbert--Finsler unit balls are uniformly comparable to a fixed reference body derived from a Macbeath region. As a consequence, local Hilbert Holmes--Thompson volume and area are comparable to the corresponding Minkowski Holmes--Thompson quantities for the associated reference norm. This section develops local geometric comparison estimates that will be used later. 

\subsection{Macbeath Regions and Hilbert Balls}

We begin by introducing Macbeath regions and present their relation to Hilbert balls. Macbeath regions provide a convenient local model for the geometry of a convex body near an interior point. For a convex body $K$ and a point $x\in \interior(K)$, the \emph{Macbeath region} at $x$ is 
\[
    M_K(x)
        ~ := ~ K \cap (2 x - K).
\]
This is a convex body that is centrally symmetric about $x$ (see Figure~\ref{f:macbeath}(a)). More generally, for $\alpha > 0$, define the scaled Macbeath region
\[
    M_K(x,\alpha)
        ~ :=~ x + \alpha \bigl( (K-x) \cap (x-K) \bigr),
\]
which is obtained from $M_K(x)$ by scaling about $x$ by a factor of $\alpha$ (see Figure~\ref{f:macbeath}(b)).

\begin{figure}[htbp]
  \centerline{\includegraphics[scale=0.40]{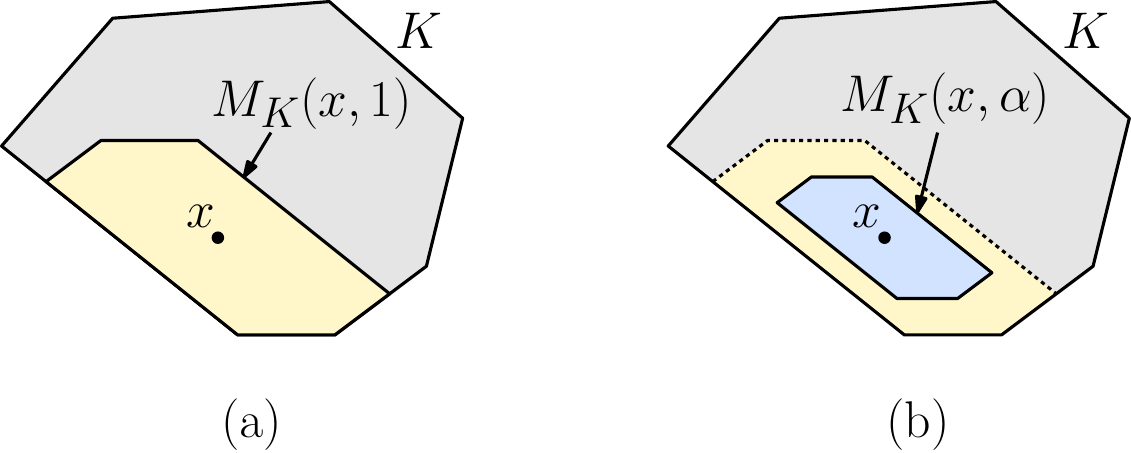}}
  \caption{Macbeath regions} \label{f:macbeath}
\end{figure}

Fix $x \in \interior(K)$ and set $B := \ballHilb[K](x,\RLocal)$. All local statements below are made on $B$ or on a smaller concentric ball $\ballHilb[K](x,r)$ with $0 < r \leq \RLocal$. 

We shall use two standard relations between Macbeath regions and bounded Hilbert balls. The first is the usual overlap-containment property for shrunken Macbeath regions~\cite{ELR70, BCP93}. The second is a bounded-radius comparison between Hilbert balls and scaled Macbeath regions~\cite{AbM18, VeW21}. 

\begin{lemma}[Macbeath expansion-containment]\label{lem:mac-containment} 
Let $x$ and $y$ be two points in a convex body $K$ such that $M_K(x,1/5)  \cap M_K(y,1/5) \neq \emptyset$. 
Then 
\[
    M_K(y, 1/5) ~ \subseteq ~ M_K(x, 1). 
\]
\end{lemma}

\begin{lemma}[Macbeath--Hilbert sandwich]\label{lem:mac-hilbert}
There exist constants $\sigma, \tau > 0$ (independent of dimension), with $\sigma < 1 \leq \tau$, such that for any convex body $K$, any $x \in \interior(K)$, and any $0 \leq r \leq \RLocal$, 
\[
    M_K(x, \sigma r)
        ~ \subseteq ~ \ballHilb[K](x,r) 
        ~ \subseteq ~ M_K(x, \tau r). 
\]
\end{lemma}

\subsection{Fixed-Reference Comparison and Local Growth}

Recall that in the Hilbert geometry, the Holmes--Thompson density at $y$ is determined by the dual unit ball
$\polar{\finsBallHilb[K](y)}$. For $x\in\interior(K)$, define 
\[
    A(x) ~ := ~ M_K(x,1) - x.
\]
The next lemma shows that on $B := \ballHilb[K](x,\RLocal)$, the local Hilbert--Finsler unit balls are uniformly comparable to the fixed reference body $A(x)$. Since $A(x)$ is centrally symmetric, it defines a Minkowski norm $\|\cdot\|_{A(x)}$.

\begin{lemma}\label{lem:local-finsler-mink} 
Let $K$ be a convex body in $\RE^d$, let $x \in \interior(K)$, and set $B := \ballHilb[K](x,\RLocal)$. There exists a constant $c_0 = c_0(\RLocal)>1$ such that for all $y\in B$, 
\[
    \frac{1}{c_0}A(x)
        ~ \subseteq ~ \finsBallHilb[K](y)
        ~ \subseteq ~ c_0 A(x).
\]
\end{lemma}

\begin{proof}
For $y\in\interior(K)$, recall that $A(y) := M_K(y,1) - y$. We first claim that
\begin{equation}\label{eq:Ay-ball}
    A(y)
        ~ \subseteq ~ \finsBallHilb[K](y)
        ~ \subseteq ~ 2 A(y).
\end{equation}
Indeed, for any $u\in S^{d-1}$, let $t^+, t^- > 0$ denote the Euclidean distances from $y$ to $\bd K$ in the directions $u$ and $-u$, respectively. Then the radial distance of $A(y)$ in direction $u$ is $\min\{t^+, t^-\}$, whereas the radial distance of $\finsBallHilb[K](y)$ is the harmonic mean of $t^+$ and $t^-$. Since the harmonic mean lies between $\min\{t^+, t^-\}$ and $2 \min\{t^+, t^-\}$, Eq.~\eqref{eq:Ay-ball} follows. 

Now fix $y\in B$. Set $\eta := 1/5$ and $\delta := \eta/\tau$, where $\tau$ is the constant from Lemma~\ref{lem:mac-hilbert}. Since $\tau \geq 1$ and $\RLocal \geq 1$, we have $\delta \leq 1/5 \leq \RLocal$, and hence Lemma~\ref{lem:mac-hilbert} implies
\[
    \ballHilb[K](u,\delta)
        ~ \subseteq ~ M_K(u,\eta),
        \text{~~for all $u \in \interior(K)$}.
\]

Choose points $z_0=x,\dots,z_m=y$ on the segment $\overline{x y}$ so that $\distHilb[K](z_i,z_{i+1}) \leq \delta$. Since Euclidean segments are Hilbert geodesics and $\distHilb[K](x,y) \leq \RLocal$, we may arrange that $m \leq \ceil{\RLocal / \delta}$. Set
\[
    m_0
        ~ := ~ \ceil{\frac{\RLocal}{\delta}},  
\]
so $m_0$ depends only on $\RLocal$.

For $i = 0, \dots, m-1$, set $u := z_i$ and $v := z_{i+1}$. Then $v \in \ballHilb[K](u,\delta) \subseteq M_K(u,\eta)$, while trivially $v \in M_K(v,\eta)$. Hence $M_K(u,\eta) \cap M_K(v,\eta) \neq \emptyset$, and by Lemma~\ref{lem:mac-containment} we have
\[
    M_K(v,\eta)
        ~ \subseteq ~ M_K(u,1).
\]
Translating by $-v$, we obtain
\[
    \eta \SP A(v)
        ~ \subseteq ~A(u) + (u-v).
\]
Since $v\in M_K(u,\eta)$, we have $v-u \in \eta A(u)$. Since $A(u)$ is centrally symmetric, we also have $u-v \in \eta A(u)$. Thus,
\[
    A(v)
        ~ \subseteq ~\frac{1+\eta}{\eta}\,A(u).
\]
Interchanging $u$ and $v$ gives the reverse inclusion. Thus, by setting $\kappa := (1+\eta)/\eta$ and iterating along the chain, we obtain
\[
    \kappa^{-m}A(x)
        ~ \subseteq ~ A(y)
        ~ \subseteq ~ \kappa^m A(x).
\]
Since $m \leq m_0$, this implies  
\[
    \kappa^{-m_0}A(x)
        ~ \subseteq ~ A(y)
        ~ \subseteq ~ \kappa^{m_0}A(x).  
\]
Let $b_0:=\kappa^{m_0}$, which depends only on $\RLocal$. Combining this with Eq.~\eqref{eq:Ay-ball} yields
\[
    b_0^{-1} A(x)
        ~ \subseteq ~ A(y)
        ~ \subseteq ~ \finsBallHilb[K](y)
        ~ \subseteq ~ 2 A(y)
        ~ \subseteq ~ 2 b_0 A(x).
\]
Setting $c_0 := 2 b_0$ completes the proof.
\end{proof}

Using this, we show that within the fixed ball $\ballHilb[K](x,\RLocal)$, Hilbert Holmes--Thompson volumes and areas are comparable to the Minkowski geometry induced by the Macbeath region centered at $x$.

\begin{lemma}[Local measure comparisons]\label{lem:local-ht-compare}
Given any convex body $K$ in $\RE^d$ and any $x \in \interior(K)$, let $B := \ballHilb[K](x,\RLocal)$ and recall $A(x) := M_K(x,1) - x$. Then:
\begin{enumerate}
\item[$(i)$] For any Lebesgue measurable $U \subseteq B$, $\volHilb[K](U) \CEQ \volMink[A(x)](U)$.
\item[$(ii)$] For any piecewise $C^1$ surface $S\subseteq B$, $\areaHilb[K](S) \CEQ \areaMink[A(x)](S)$.
\end{enumerate}
\end{lemma}

\begin{proof}
Let $c_0 = c_0(\RLocal) > 1$ denote the constant from Lemma~\ref{lem:local-finsler-mink}.
Then, for all $y\in B$,
\[
    \frac{1}{c_0} A(x)
        ~ \subseteq ~ \finsBallHilb[K](y)
        ~ \subseteq ~ c_0 A(x). 
\]
Taking polars yields
\[
    \frac{1}{c_0} \polar{A(x)}
        ~ \subseteq ~ \polar{\finsBallHilb[K](y)}
        ~ \subseteq ~ c_0 \polar{A(x)}, 
\]
and hence $\lambda_d \bigl( \polar{\finsBallHilb[K](y)} \bigr) \CEQ \lambda_d \bigl( \polar{A(x)} \bigr)$. Therefore,
\[
    \volHilb[K](U)
        ~ =    ~ \int_U \frac{1}{\omega_d} \lambda_d \bigl( \polar{\finsBallHilb[K](y)} \bigr) d\lambda_d(y)
        ~ \CEQ ~ \int_U \frac{1}{\omega_d} \lambda_d \bigl( \polar{A(x)} \bigr) d\lambda_d(y)
        ~ =    ~ \volMink[A(x)](U),
\]
which establishes~(i).

To prove~(ii), let $S\subseteq B$ be piecewise $C^1$. For $\lambda_{d-1}$-a.e\ $y \in S$, the tangent space $T_y$ is well defined. Intersecting the above inclusion with $T_y$ and taking polars in $T_y$ gives
\[
    \frac{1}{c_0}\,\polarIn[T_y]{(A(x)\cap T_y)}
        ~ \subseteq ~ \polarIn[T_y]{(\finsBallHilb[K](y)\cap T_y)}
        ~ \subseteq ~ c_0\,\polarIn[T_y]{(A(x)\cap T_y)}.
\]
Thus,
\[
    \lambda_{d-1} \bigl( \polarIn[T_y]{(\finsBallHilb[K](y)\cap T_y)} \bigr)
        ~ \CEQ ~ \lambda_{d-1} \bigl( \polarIn[T_y]{(A(x)\cap T_y)} \bigr).
\]
Integrating yields $\areaHilb[K](S) \CEQ \areaMink[A(x)](S)$, as desired.
\end{proof}

The next two lemmas present the basic Holmes--Thompson growth estimates for metric balls in the Minkowski and Hilbert settings. In the Minkowski case, the estimates hold at all scales, whereas in the Hilbert case, they are restricted to bounded radii.

We begin with the Minkowski case, which will serve as the fixed-norm input for the Hilbert estimate below. The area bound follows from Thompson~\cite[Theorems~6.5.2 and~6.5.3]{Tho96}, while the volume bound follows from the Holmes--Thompson volume-product formula~\cite[Section~3.2]{AlT04}, together with the Blaschke--Santal\'o and reverse Santal\'o inequalities~\cite{San49, BoM87, MeP90}. 

\begin{lemma} \label{lem:mink-ball-growth}
For any centrally symmetric convex body $A$ in $\RE^d$, any $r \geq 0$, and any $x \in \RE^d$:
\begin{enumerate}
\item[$(i)$] $\volMink[A](\ballMink[A](x,r)) \;\CEQ\; \omega_d\,r^d$ and
\item[$(ii)$] $\areaMink[A](\ballMink[A](x,r)) \;\CEQ\; \omega_{d-1}\,r^{d-1}$.
\end{enumerate}
\end{lemma}

Using the fixed-reference comparison with Lemma~\ref{lem:mink-ball-growth}, we obtain the corresponding bounded-radius estimate in Hilbert geometry.

\begin{lemma} \label{lem:hilb-ball-growth}
For any convex body $K$ in $\RE^d$, any $x \in \interior(K)$, and any $0 < r \leq \RLocal$:
\begin{enumerate}
\item[$(i)$] $\volHilb[K](\ballHilb[K](x,r)) \;\CEQ\; \omega_d\,r^d$ and
\item[$(ii)$] $\areaHilb[K](\ballHilb[K](x,r)) \;\CEQ\; \omega_{d-1}\,r^{d-1}$.
\end{enumerate}
\end{lemma}

\begin{proof}
Lemma~\ref{lem:mac-hilbert} implies that
\[
    x + \sigma r A(x) 
        ~ \subseteq ~ \ballHilb[K](x,r) 
        ~ \subseteq ~ x + \tau r A(x),
\]
for suitable constants $\sigma, \tau > 0$. Applying Lemma~\ref{lem:local-ht-compare}, together with the monotonicity of Holmes--Thompson volumes and areas (Lemma~\ref{lem:HT-area-monotone}), yields
\[
    \volHilb[K](\ballHilb[K](x,r)) ~ \CEQ ~ \volMink[A(x)](x + r A(x))
        \quad\text{and}\quad
    \areaHilb[K](\ballHilb[K](x,r)) ~ \CEQ ~ \areaMink[A(x)](x + r A(x)).
\]
The result now follows from Lemma~\ref{lem:mink-ball-growth}.
\end{proof}

\section{Polarity-Expansion Stability} \label{s:distortion}

In our duality arguments, we thicken a convex body by including all points within distance $\alpha$, which we call its \emph{$\alpha$-expansion}. For fixed $\alpha$, we write $G_+$ for the corresponding expansion (the precise definitions in the Minkowski and Hilbert settings are given below). The key issue is that, after passing to polars, this thickening changes the ambient geometry. In the polar setting, the geometry is induced by $\polar{G_+}$, whereas the target statements are formulated in the geometry of $\polar{G}$. The purpose of this section is to control this mismatch on the common domain $\polar{K}$. More precisely, in both the Minkowski and Hilbert settings, we prove the stability bound:
\[
    \distX[\polar{G_+}](p,x)
        ~ \leq ~ \distX[\polar{G}](p,x) + 2 \alpha,
\]
for all $p,x\in \polar{K}$. Hence, every radius-$\alpha$ ball in the geometry of $\polar{G}$ induces a radius-$3 \alpha$ ball in the geometry of $\polar{G_+}$ on $\polar{K}$. 

Given a generic distance function $\text{dist}(\cdot, \cdot)$, we define the distance to any set $G$ in the natural way as
\[
    \text{dist}(x,G)
        ~ := ~ \inf_{g \in G} \text{dist}(x,g).
\]
For $\alpha > 0$, define the \emph{Minkowski $\alpha$-expansion} of $G$ (with respect to $K$) by
\[
    \expandM[K](G,\alpha)
        ~ := ~ \{x\in \RE^d \ST \distMink[K](x,G) \leq \alpha\}.
\]
Equivalently, $\expandM[K](G, \alpha) = G + \alpha K$ (where ``$+$'' denotes Minkowski addition here). 

Next, consider the Hilbert geometry induced by a body $K$. Given a convex body $G \subset \interior(K)$ and $\alpha > 0$, we define the \emph{Hilbert $\alpha$-expansion} of $G$ in $K$, denoted $\expandH[K](G, \alpha)$, analogously, but using Hilbert distances. 

\subsection{Minkowski Distortion Bound} \label{s:distortion-mink}

Let $C$ and $D$ be centrally symmetric convex bodies in $\RE^d$, and consider the Minkowski geometry induced by $D$. We begin with a technical observation showing that, in the Minkowski setting, polarity interacts with Minkowski sums via an explicit gauge identity.

\begin{lemma} \label{lem:polar-of-sum-mink} 
For any centrally symmetric convex bodies $C$ and $D$ in $\RE^d$, any $\alpha > 0$, and any $u \in \RE^d$, 
\[
    \|u\|_{\polar{(C + \alpha D)}}
        ~ = ~ \|u\|_{\polar{C}} + \alpha\,\|u\|_{\polar{D}}.
\]
\end{lemma}

\begin{proof}
For a centrally symmetric convex body $C$, the gauge of $\polar{C}$ equals the support function of $C$, that is, $\|u\|_{\polar{C}} = h_C(u)$. Since support functions are additive under Minkowski sums, $h_{C + \alpha D} = h_C + \alpha h_D$, and hence 
\[
    \|u\|_{\polar{(C + \alpha D)}}
        ~ = ~ h_{C + \alpha D}(u)
        ~ = ~ h_C(u) + \alpha h_D(u)
        ~ = ~ \|u\|_{\polar{C}} + \alpha\,\|u\|_{\polar{D}}.
\]
\end{proof}

Using this, we can relate the dual norms induced by $\polar{C}$ and $\polar{C_+}$ on the common domain $\polar{D}$.

\begin{lemma}[Minkowski polarity-expansion stability]\label{lem:polar-exp-mink}
Let $C$ and $D$ be centrally symmetric convex bodies in $\RE^d$, let $\alpha > 0$, and let $C_+ := \expandM[D](C,\alpha)$. Then, for all $p,x \in \polar{D}$,
\[
    \distMink[\polar{C_+}](p,x)
        ~ \leq ~ \distMink[\polar{C}](p,x) + 2 \alpha.
\]
In particular, if $\distMink[\polar{C}](p,x) \leq \alpha$, then $\distMink[\polar{C_+}](p,x) \leq 3 \alpha$. 
\end{lemma}

\begin{proof}
Let $v := p-x$, which implies that $\distMink[\polar{C}](p,x) = \|v\|_{\polar{C}}$. Since $p,x \in \polar{D}$ and $\polar{D}$ is centrally symmetric and convex, we have $v \in \polar{D} - \polar{D} = 2\polar{D}$, and hence $\|v\|_{\polar{D}} \leq 2$. Since $\expandM[D](C, \alpha) = C + \alpha D$, Lemma~\ref{lem:polar-of-sum-mink} implies that
\[
    \distMink[\polar{C_+}](p,x)
        ~ = ~ \|v\|_{\polar{C_+}}
        ~ = ~ \|v\|_{\polar{C}} + \alpha\,\|v\|_{\polar{D}}
        ~ \leq ~ \distMink[\polar{C}](p,x) + 2 \alpha,
\]
as desired.
\end{proof}

This distortion bound implies that the size of any $\alpha$-covering in the modified geometry is roughly the same as in the original geometry.

\begin{lemma}\label{lem:polar-exp-cover-mink}
Given the same preconditions as Lemma~\ref{lem:polar-exp-mink}, for any set $U \subseteq \polar{D}$,
\[
    \NcovMink[\polar{C_+}](U, \alpha)
        ~ \CLEQ ~ \NcovMink[\polar{C}](U, \alpha).
\] 
\end{lemma}

\begin{proof}
Let $m=\NcovMink[\polar{C}](U,\alpha)$, and fix an $\alpha$-cover $X = \{x_1, \dots, x_m\}$, so that $U \subseteq \bigcup_{i=1}^m \ballMink[\polar{C}](x_i,\alpha)$. For each $i$, choose a point $p_i \in U \cap \ballMink[\polar{C}](x_i, \alpha)$ whenever this intersection is nonempty, and discard the remaining balls. Now fix $i$ and let $x \in U \cap \ballMink[\polar{C}](x_i, \alpha)$. Since both $p_i$ and $x$ lie in $\ballMink[\polar{C}](x_i,\alpha)$, the triangle inequality gives
\[
    \distMink[\polar{C}](p_i,x) ~ \leq ~ 2 \alpha.
\]
Because $p_i,x \in U \subseteq \polar{D}$, Lemma~\ref{lem:polar-exp-mink} implies
\[
    \distMink[\polar{C_+}](p_i,x)
        ~ \leq ~ \distMink[\polar{C}](p_i, x) + 2 \alpha
        ~ \leq ~ 4 \alpha.
\]
Thus, $U \subseteq \bigcup_{i=1}^m \ballMink[\polar{C_+}](p_i, 4 \alpha)$, and hence,
\[
    \NcovMink[\polar{C_+}](U, 4 \alpha)
        ~ \leq ~ m
        ~ =    ~ \NcovMink[\polar{C}](U, \alpha).
\]
Applying Lemma~\ref{lem:mink-ball-cover} with $D = \polar{C_+}$, $r = \alpha$, and $r' = 4 \alpha$, we obtain
\[
    \NcovMink[\polar{C_+}](U, \alpha)
        ~ \CLEQ ~ \NcovMink[\polar{C_+}](U, 4 \alpha)
        ~ \leq  ~ \NcovMink[\polar{C}](U, \alpha),
\]
as desired.
\end{proof}

\bigskip

Notably, the factor $3$ of Lemma~\ref{lem:polar-exp-mink} is sharp. To see why, consider $d=1$, let $D = [-1,1]$, and for any $\alpha > 0$, let $C = [-\alpha/2, \alpha/2]$. Since $C_+ = C + \alpha D$, we have
\[
    C_+ 
        ~ = ~ \left[ -\frac{3\alpha}{2}, \frac{3\alpha}{2} \right],
    \quad
    \polar{C} 
        ~ = ~ \left[ \frac{-2}{\alpha}, \frac{2}{\alpha} \right]
    \quad\text{and}\quad
    \polar{C_+}
        ~ = ~ \left[ \frac{-2}{3 \alpha}, \frac{2}{3 \alpha} \right].
\]
Setting $p=1$ and $x=-1$, we have
\[
    \distMink[\polar{C}](p,x)
        ~ = ~ \|2\|_{\polar{C}}
        ~ = ~ \alpha
    \quad\text{and}\quad
    \distMink[\polar{C_+}](p,x)
        ~ = ~ \|2\|_{\polar{C_+}}
        ~ = ~ 3 \alpha.
\]
Thus, the distortion $3$ is attained already in dimension $1$.

\subsection{Hilbert Distortion Bound} \label{s:distortion-hilb}

A basic complication in the Hilbert setting is that Hilbert expansions do not commute with polarity in the sense that if $G_+$ denotes the Hilbert $\alpha$-expansion of $G$ in $K$, the polar body $\polar{G_+}$ need not be a controlled Hilbert thickening of $\polar{G}$. The following lemma provides a robust substitute that is uniform over all $p, x \in \polar{K}$. 

\begin{lemma} \label{lem:polar-exp-hilb}
Let $G$ and $K$ be convex bodies in $\RE^d$ with $O \in \interior(G)$ and $G \subset \interior(K)$. Let $\alpha > 0$, and let $G_+ := \expandH[K](G,\alpha)$. Then, for all $p,x \in \polar{K}$,
\[
    \distHilb[\polar{G_+}](p,x)
        ~ \leq ~ \distHilb[\polar{G}](p,x) + 2\alpha.
\]
In particular, if $\distHilb[\polar{G}](p,x) \leq \alpha$, then $\distHilb[\polar{G_+}](p,x) \leq 3\alpha$.
\end{lemma}

\begin{proof}
We make use of the forward Funk distance on a convex body $C$ with $O \in \interior(C)$. Recalling that Funk distance is denoted by $\distFunk[C](\cdot, \cdot)$,  we use the standard variational formula for Funk geometry~\cite[Corollary~2.6]{PaT14a}, rewritten in polar form, that for all $u,v \in \interior(C)$
\begin{equation}\label{eq:funk}
    \distFunk[C](u,v)
        ~ = ~ \log \left( \sup_{z\in \polar{C}} \frac{1-\inner{z}{u}}{1-\inner{z}{v}} \right).
\end{equation}
Since the Hilbert distance is the arithmetic mean of the forward and reverse Funk distances, we have
\begin{equation}\label{eq:hilb-funk}
    \distHilb[C](u,v)
        ~ = ~ \frac{1}{2} \left( \distFunk[C](u,v) + \distFunk[C](v,u) \right).
\end{equation}

Since $G_+\subset \interior(K)$ and polarity reverses inclusion, we have $\polar{K} \subset \interior(\polar{G_+})$, so $\distHilb[\polar{G_+}](p,x)$ is well-defined for all $p,x \in \polar{K}$. Given $p,x \in \polar{K}$, applying Eq.~\eqref{eq:funk} with $C = \polar{G_+}$ yields 
\begin{equation}\label{eq:funk-polarGplus}
    \distFunk[\polar{G_+}](p,x)
        ~ = ~ \log \left( \sup_{z\in G_+} \frac{1-\inner{z}{p}}{1-\inner{z}{x}} \right).
\end{equation}
Let $z \in G_+$ and choose $g \in G$ minimizing $\distHilb[K](g,z)$. Then $\distHilb[K](g,z) = \distHilb[K](z,G) \leq \alpha$. Thus,
\begin{equation}\label{eq:funk-sum}
    \distFunk[K](z,g) + \distFunk[K](g,z)
        ~ =    ~ 2 \cdot \distHilb[K](g,z)
        ~ \leq ~ 2 \alpha.
\end{equation}

We next compare the numerator and denominator using the corresponding Funk terms. By definition of $\distFunk[K]$, for $p \in \polar{K}$,
\[
    \log \left( \frac{1-\inner{p}{z}}{1-\inner{p}{g}} \right)
        ~ \leq ~ \distFunk[K](z,g)
    \quad\Longrightarrow\quad
    1 - \inner{z}{p} 
        ~ \leq ~ \exp \bigl( \distFunk[K](z,g) \bigr) \bigl( 1-\inner{g}{p} \bigr),
\]
and similarly, for $x\in \polar{K}$,
\[
    \log \left( \frac{1-\inner{x}{g}}{1-\inner{x}{z}}\right)
        ~ \leq ~ \distFunk[K](g,z)
    \quad\Longrightarrow\quad
    1-\inner{z}{x} 
        ~ \geq ~ \exp\bigl( -\distFunk[K](g,z) \bigr) \bigl( 1-\inner{g}{x} \bigr).
\]
Combining these inequalities yields
\[
    \frac{1-\inner{z}{p}}{1-\inner{z}{x}}
        ~ \leq ~ \exp \bigl( \distFunk[K](z,g) + \distFunk[K](g,z) \bigr) \frac{1-\inner{g}{p}}{1-\inner{g}{x}}
        ~ \leq ~ \exp(2 \alpha) \frac{1-\inner{g}{p}}{1-\inner{g}{x}},
\]
where the last step uses Eq.~\eqref{eq:funk-sum}. Taking the supremum over $z\in G_+$ and then over $g\in G$, and using Eq.~\eqref{eq:funk-polarGplus}, we obtain 
\[
    \distFunk[\polar{G_+}](p,x)
        ~ \leq ~ 2 \alpha + \log \left( \sup_{g\in G} \frac{1-\inner{g}{p}}{1-\inner{g}{x}} \right).
\]
By Eq.~\eqref{eq:funk}, the logarithmic term equals $\distFunk[\polar{G}](p,x)$. Therefore, $\distFunk[\polar{G_+}](p,x) \leq \distFunk[\polar{G}](p,x) + 2 \alpha$. By the same argument with $p$ and $x$ interchanged, we have $\distFunk[\polar{G_+}](x,p) \leq \distFunk[\polar{G}](x,p) + 2 \alpha$. Averaging these two inequalities and applying Eq.~\eqref{eq:hilb-funk} gives 
\[
    \distHilb[\polar{G_+}](p,x)
        ~ \leq ~ \distHilb[\polar{G}](p,x) + 2 \alpha,
\]
as desired. 
\end{proof}

This distortion bound implies that, at bounded scales, the size of any $\alpha$-covering in the modified geometry is roughly the same as in the original geometry.

\begin{lemma}\label{lem:polar-exp-cover-hilb} 
Given the same preconditions as Lemma~\ref{lem:polar-exp-hilb}, and assuming $0 < \alpha \leq \RLocal/8$, for any set $U \subseteq \polar{K}$, 
\[
    \Ncov[\polar{G_+}](U,\alpha)
        ~ \CLEQ ~ \Ncov[\polar{G}](U,\alpha).
\] 
\end{lemma} 

\begin{proof} 
The proof is the same as that of Lemma~\ref{lem:polar-exp-cover-mink}, with Hilbert balls and Lemma~\ref{lem:polar-exp-hilb} in place of their Minkowski counterparts. The only additional point is that the final application of Lemma~\ref{lem:hilb-ball-cover}, with $r = \alpha$ and $r' = 4 \alpha$, is valid under the assumption $\alpha \leq \RLocal/8$. 
\end{proof} 

\bigskip

As in the Minkowski case, the constant $3$ in Lemma~\ref{lem:polar-exp-hilb} is sharp, even when $G$ and $K$ are centrally symmetric and already in dimension $1$. To see this, let $K = [-1,1]$, let $G = [-a,a]$ with $a = \tanh(\alpha/2)$, and let $p = 1$, $x = -1$. It is easily verified that $\distHilb[\polar{G}](p,x) = \alpha$. The Hilbert $\alpha$-expansion of $G$ in $K$ is $G_+ = [-\tanh(3\alpha/2),\tanh(3\alpha/2)]$. Again, it is easily verified that $\distHilb[\polar{G_+}](p,x) = 3 \alpha$. 

\section{Expansion Properties for Volumetric Coverings} \label{s:prop-exp-vol}

In our duality arguments, we repeatedly pass from a body to its $\alpha$-expansion. In this section, we present basic properties of these expansions, which will be used later in our analysis of volumetric coverings.

We begin with the Minkowski setting. Given convex bodies $C$ and $D$ in $\RE^d$, with $D$ centrally symmetric, and $\alpha > 0$, recall that the Minkowski $\alpha$-expansion of $C$ is $\expandM[D](C,\alpha) = C + \alpha D$. The following lemma is a straightforward consequence of properties of Minkowski addition.

\begin{lemma}\label{lem:mink-exp-i}
Let $C$ and $D$ be convex bodies in $\RE^d$ where $D$ is centrally symmetric. Let $\alpha > 0$ and let $C_+ := \expandM[D](C,\alpha)$. Then $C_+$ is convex and $C \subset \interior(C_+)$.
\end{lemma}

Our next lemma relates the volumetric Minkowski covering number of $C$ to the Holmes--Thompson volume of $C$ and its expansion.

\begin{lemma}\label{lem:mink-exp-ii}
Let $C$, $D$, and $\alpha$ be as in Lemma~\ref{lem:mink-exp-i}. Then
\[
    \frac{\volMink[D](C)}{\omega_d \, \alpha^d} ~ \CLEQ ~ \NcovMink[D](C,\alpha) ~ \CLEQ ~ \frac{\volMink[D](C_+)}{\omega_d \, \alpha^d}.
\]
\end{lemma}

\begin{proof}
Let $X \subset C$ be maximal $\alpha$-separated in $\distMink[D]$. Then $\{\ballMink[D](x,\alpha)\}_{x\in X}$ covers $C$, while $\{\ballMink[D](x, \frac{\alpha}{2})\}_{x \in X}$ have pairwise disjoint interiors and are contained in $C + \alpha D = C_+$. Hence, by Lemma~\ref{lem:mink-ball-growth}(i), 
\[
    \NcovMink[D](C,\alpha)
        ~ \leq ~ |X|
    \quad\text{and}\quad
    |X|\,\omega_d\, \left( \frac{\alpha}{2} \right)^d
        ~ \CLEQ ~ \sum_{x\in X} \volMink[D] \left( \ballMink[D] \left( x,\frac{\alpha}{2} \right) \right)
        ~ \leq  ~ \volMink[D](C_+),
\]
and hence
\[
    \NcovMink[D](C,\alpha)
        ~ \CLEQ ~ \frac{\volMink[D](C_+)}{\omega_d\,\alpha^d},
\]
establishing the second inequality.

Conversely, given any discrete set $X$ such that $\{\ballMink[D](x,\alpha)\}_{x \in X}$ covers $C$, then by subadditivity and Lemma~\ref{lem:mink-ball-growth}(i), 
\[
    \volMink[D](C)
        ~ \leq ~ \sum_{x \in X} \volMink[D]\bigl(\ballMink[D](x,\alpha)\bigr)
        ~ \CLEQ ~ |X| \, \omega_d \, \alpha^d.
\]
Taking the infimum over all such covers yields
\[
    \frac{\volMink[D](C)}{\omega_d \, \alpha^d}
        ~ \CLEQ ~ \NcovMink[D](C,\alpha), 
\]
which establishes the first inequality.
\end{proof}

\bigskip

Next, we turn to Hilbert expansions. Recall that for convex bodies $G$ and $K$ in $\RE^d$ such that $G \subset \interior(K)$ and $\alpha > 0$, the Hilbert $\alpha$-expansion of $G$ in $K$ is defined by 
\[
    \expandH[K](G,\alpha)
        ~ := ~ \{x \in \interior(K) \ST \distHilb[K](x,G) \leq \alpha\}. 
\]

The first lemma states a basic geometric property of this expansion.

\begin{lemma}\label{lem:hilb-exp-i} 
Let $G$ and $K$ be convex bodies in $\RE^d$ with $G \subset \interior(K)$, and let $0 < \alpha \leq 1$. Let $G_+ := \expandH[K](G,\alpha)$. Then $G_+$ is a convex body and $G \subset \interior(G_+) \subset \interior(K)$.
\end{lemma}

\begin{proof}
This is a standard property of Hilbert geometries (see, e.g., Busemann~\cite[Prop.~18.9]{Bus55}). Also, $G \subset \interior(G_+)$ since $\alpha > 0$. 
\end{proof}

Our next lemma relates the volumetric Hilbert covering number of $G$ to the Holmes--Thompson volume of $G$ and its expansion. 

\begin{lemma}\label{lem:hilb-exp-ii} 
Let $G$, $K$, and $\alpha$ be as in Lemma~\ref{lem:hilb-exp-i}. Then 
\[
    \frac{\volHilb[K](G)}{\omega_d \, \alpha^d}
        ~ \CLEQ ~ \Ncov[K](G,\alpha)
        ~ \CLEQ ~ \frac{\volHilb[K](G_+)}{\omega_d \, \alpha^d}. 
\]
\end{lemma}

\begin{proof}
The proof follows by the same maximal-separated-set argument as in the proof of Lemma~\ref{lem:mink-exp-ii}: a maximal $\alpha$-separated subset of $G$ yields a cover of $G$ by radius-$\alpha$ balls and a disjoint packing by radius-$\alpha/2$ balls inside $G_+$. The only change is that we estimate the volumes of these balls using Lemma~\ref{lem:hilb-ball-growth}(i).
\end{proof}

\section{Duality for Volumetric Coverings} \label{s:volume-cover}

In this section, we provide proofs for the duality of volumetric coverings. As a warm-up exercise, we begin with the (known) duality for volumetric coverings by translates in a Minkowski space.

\begin{lemma}[Duality for volumetric coverings by translates] \label{lem:volume-cover-duality-mink}
There exists an absolute constant $c \geq 1$ such that for any pair of centrally symmetric convex bodies $C$ and $D$, and any $\alpha > 0$,
\[
    c^{-d} \, \NcovMink[\polar{C}](\polar{D}, \alpha)
        ~ \leq ~ \NcovMink[D](C, \alpha)
        ~ \leq ~ c^{d} \, \NcovMink[\polar{C}](\polar{D}, \alpha).
\]
\end{lemma}

\begin{proof}
Fix $\alpha > 0$ and set $C_+ = \expandM[D](C,\alpha) = C + \alpha D$. By Lemma~\ref{lem:mink-exp-i}, $C_+$ is a centrally symmetric convex body and $C \subset \interior(C_+)$. By the upper-bound part of Lemma~\ref{lem:mink-exp-ii} and the Holmes--Thompson polarity identity (Lemma~\ref{lem:HT-mink-polarity}(i))
\begin{equation}\label{eq:volume-cover-duality-mink-1}
    \NcovMink[D](C, \alpha)
        ~ \CLEQ ~ \frac{\volMink[D](C_+)}{\omega_d \, \alpha^d}
        ~ =     ~ \frac{\volMink[\polar{C_+}](\polar{D})}{\omega_d \, \alpha^d}.
\end{equation}
Applying the lower-bound part of Lemma~\ref{lem:mink-exp-ii} in the Minkowski norm with unit ball $\polar{C_+}$ gives
\begin{equation}\label{eq:volume-cover-duality-mink-2}
    \frac{\volMink[\polar{C_+}](\polar{D})}{\omega_d \, \alpha^d}
        ~ \CLEQ ~ \NcovMink[\polar{C_+}](\polar{D}, \alpha).
\end{equation}
Applying Lemma~\ref{lem:polar-exp-cover-mink} with $U = \polar{D}$, we obtain
\begin{equation}\label{eq:volume-cover-duality-mink-3}
    \NcovMink[\polar{C_+}](\polar{D}, \alpha)
        ~ \CLEQ ~ \NcovMink[\polar{C}](\polar{D}, \alpha).
\end{equation} 
By chaining Eqs.~\eqref{eq:volume-cover-duality-mink-1}--\eqref{eq:volume-cover-duality-mink-3} together, we obtain
\[
    \NcovMink[D](C,\alpha)
        ~ \CLEQ ~ \NcovMink[\polar{C}](\polar{D}, \alpha),
\]
which establishes the second inequality. To establish the first inequality, apply the same argument with $(C, D)$ replaced by $(\polar{D}, \polar{C})$.
\end{proof}

\bigskip

By applying a similar approach, we extend this to the volumetric coverings in the Hilbert geometry. This proves the first part of Theorem~\ref{thm:hilbert-cover-duality}, which we restate below.

\thmHilbertCoverDuality*

\begin{proof} (Of part~(i).)
Fix $\alpha \in (0, 1]$ and set $G_+ = \expandH[K](G,\alpha)$. By Lemma~\ref{lem:hilb-exp-i}, $G_+$ is a convex body with $G \subset \interior(G_+) \subset \interior(K)$. By the upper-bound part of Lemma~\ref{lem:hilb-exp-ii} and the Holmes--Thompson polarity identity (Lemma~\ref{lem:HT-hilb-polarity}(i))
\begin{equation}\label{eq:volume-cover-duality-hilb-1}
    \Ncov[K](G, \alpha)
        ~ \CLEQ ~ \frac{\volHilb[K](G_+)}{\omega_d \, \alpha^d}
        ~ \CEQ    ~ \frac{\volHilb[\polar{G_+}](\polar{K})}{\omega_d \, \alpha^d}.
\end{equation}
Applying the lower-bound part of Lemma~\ref{lem:hilb-exp-ii} in the Hilbert geometry induced by $\polar{G_+}$, gives
\begin{equation}\label{eq:volume-cover-duality-hilb-2}
    \frac{\volHilb[\polar{G_+}](\polar{K})}{\omega_d \, \alpha^d}
        ~ \CLEQ ~ \Ncov[\polar{G_+}](\polar{K}, \alpha).
\end{equation}
Since $\alpha \leq 1$ and Section~\ref{s:covering} fixes $\RLocal \geq 8$, we have $\alpha \leq \RLocal/8$. Hence Lemma~\ref{lem:polar-exp-cover-hilb} applies with $U = \polar{K}$, yielding 
\begin{equation}\label{eq:volume-cover-duality-hilb-3}
    \Ncov[\polar{G_+}](\polar{K}, \alpha)
        ~ \CLEQ ~ \Ncov[\polar{G}](\polar{K}, \alpha).
\end{equation} 
By chaining Eqs.~\eqref{eq:volume-cover-duality-hilb-1}--\eqref{eq:volume-cover-duality-hilb-3} together, we obtain
\[
    \Ncov[K](G,\alpha)
        ~ \CLEQ ~ \Ncov[\polar{G}](\polar{K}, \alpha),
\]
which establishes the second inequality. To establish the first inequality, apply the same argument with $(G, K)$ replaced by $(\polar{K}, \polar{G})$.
\end{proof}

\section{Relative Isoperimetric Inequalities} \label{s:rel-iso}

Having established the duality results for volumetric coverings, we now turn to boundary coverings. A key analytic tool for this case is a pair of localized relative isoperimetric inequalities. The standard isoperimetric inequality can be used to provide a lower bound on the surface area of a body in terms of its volume. In contrast, for a given body $B$ and a cutting surface, a relative isoperimetric inequality provides a lower bound on the surface area of the cut in terms of the volume it bounds, both relative to the surface area and volume of $B$.

Our first result applies to the Minkowski (translative) case. It is formulated for an arbitrary Minkowski ball in the norm whose unit ball is centrally symmetric. It will be applied in our analysis of Minkowski $\alpha$-expansions for boundary coverings (see Lemma~\ref{lem:mink-exp-fat} below).

\begin{restatable}[Relative isoperimetry for Minkowski balls]{lemma}{lemMinkHTiso} \label{lem:mink-HT-iso}
Given any centrally symmetric convex body $D$ in $\RE^d$, any $z \in \RE^d$, and $r > 0$, let $B = \ballMink[D](z,r)$. Let $E$ be any convex body in $\RE^d$ such that
\[
    0
        ~ < ~ \volMink[D] \bigl( E \cap B \bigr)
        ~ < ~ \volMink[D](B).
\]
Define
\[
    \mu
        ~ := ~ \frac{\volMink[D] \bigl( E \cap B \bigr)}{\volMink[D](B)}
        \qquad\text{and}\qquad
    \beta
        ~ := ~ \frac{\areaMink[D] \bigl( \bd E \cap \interior(B) \bigr)}{\areaMink[D](B)}.
\]
Then $\beta \CGEQ \min(\mu, 1-\mu)$.
\end{restatable}

We next consider the Hilbert case, for balls in the small-radius regime. Here, the boundary of the cutting body is assumed to pass through the center of the ball. This is the form that will be needed in our Hilbert boundary-covering arguments (see Lemma~\ref{lem:hilb-exp-surf} below).

\begin{restatable}[Relative isoperimetry for Hilbert balls]{lemma}{lemHilbertHTiso} \label{lem:hilbert-HT-iso}
Given any convex body $K$ in $\RE^d$, any $x \in \interior(K)$, and $0 < r \leq 1$, let $B = \ballHilb[K](x,r)$. Given any convex body $E \subset \interior(K)$ with $x \in \bd E$, define
\[
    \mu
        ~ := ~ \frac{\volHilb[K] \bigl( E \cap B \bigr)}{\volHilb[K](B)}
        \qquad\text{and}\qquad
    \beta
        ~ := ~ \frac{\areaHilb[K] \bigl( \bd E \cap \interior(B) \bigr)}{\areaHilb[K](B)}.
\]
Then $\beta \CGEQ \mu$.
\end{restatable}

The proofs of Lemma~\ref{lem:mink-HT-iso} and Lemma~\ref{lem:hilbert-HT-iso} are rather technical and have been deferred to Appendix~\ref{app:rel-iso}. They proceed by first establishing a Euclidean relative isoperimetric inequality for convex cuts of an isotropic convex body,  transferring it to the Busemann normalization on a Minkowski unit ball, and then passing it to the Holmes--Thompson setting. Lemma~\ref{lem:mink-HT-iso} then follows by translation and scaling. To prove Lemma~\ref{lem:hilbert-HT-iso}, we transfer from the Minkowski ball to the Hilbert ball using the Macbeath-region sandwich and local comparability of Hilbert and Minkowski Holmes--Thompson densities.

\section{Expansion Properties for Boundary Coverings} \label{s:prop-exp-surf}

For boundary coverings, a direct bound in terms of Holmes--Thompson surface area is generally not possible at all scales $\alpha$. The main difficulty arises in the packing arguments, since the portion of the boundary covered may be arbitrarily small compared to the surface area of the covering ball. This issue can arise, for example, with skinny, needle-like objects. The purpose of expansion is to eliminate this effect by fattening the body, without significantly increasing the covering numbers. 

To formalize this concept, we need a notion of fatness that is relative to the covering element, which we call relative fatness. Let us consider a generic setting. Fix a metric-measure structure on a domain $\Omega$ with distance $\dist(\cdot,\cdot)$ and balls $\ballX(x,r)$. Assume also that we are given a volume functional $\vol(\cdot)$ on measurable subsets of $\Omega$ and a surface-area functional $\area(\cdot)$ on piecewise $C^1$ surfaces in $\Omega$. For $\alpha > 0$ and $0 < \gamma \leq 1$, we say that a measurable set $E \subseteq \Omega$ is \emph{relatively $(\alpha,\gamma)$-fat} if for every $x \in \bd E$ and every $0 < r \leq \alpha$,
\[
    \frac{\vol(E\cap \ballX(x,r))}{\vol(\ballX(x,r))} 
        ~ \geq ~ \gamma.
\]
Likewise, for $\alpha > 0$ and $0 < \gamma \leq 1$, we say that a piecewise $C^1$ surface $S \subseteq \Omega$ is \emph{relatively $(\alpha,\gamma)$-surface-fat} if for every $x \in S$ and every $0 < r \leq \alpha$,
\[
    \frac{\area(S\cap \interior(\ballX(x,r)))}{\area(\ballX(x,r))} 
        ~ \geq ~ \gamma.
\]
When $S = \bd G$ for a convex body $G \subseteq \Omega$, we also say that $G$ is \emph{relatively $(\alpha,\gamma)$-surface-fat}. When $\gamma \CGEQ 1$, we simply say that the body is \emph{relatively $\alpha$-fat} or \emph{relatively $\alpha$-surface-fat}.

\subsection{Boundary Properties in the Minkowski Setting} \label{s:prop-mink-exp-surf}

In this section, we explore the relevant properties of expansion in the Minkowski setting. Given convex bodies $C$ and $D$ in $\RE^d$, and $\alpha > 0$, recall that the Minkowski $\alpha$-expansion of $C$ is $\expandM[D](C,\alpha) = C + \alpha D$. The first property is well known and follows directly from a supporting-hyperplane argument.

\begin{lemma}\label{lem:mink-exp-bdry}
Let $C$ and $D$ be convex bodies in $\RE^d$, where $D$ is centrally symmetric. For any $\alpha > 0$, let $C_+ := \expandM[D](C,\alpha)$. Then for every $x \in \bd C$, there exists $y \in \bd C_+$ such that $\distMink[D](x,y) = \alpha$. In particular,
\[
    \bd C
        ~ \subseteq ~ \{z \ST \distMink[D](z,\bd C_+) \leq \alpha\}.
\]
\end{lemma}

Our next result implies that in the Minkowski setting, $\alpha$-expanded bodies are relatively $\alpha$-fat.

\begin{lemma}\label{lem:mink-exp-fat}
Let $C$, $D$, $\alpha$, and $C_+$ be as in Lemma~\ref{lem:mink-exp-bdry}. Then $C_+$ is relatively $\alpha$-fat and $\alpha$-surface-fat. More precisely, for any $z \in \bd C_+$, any $0 < r \leq \alpha$, and letting $B := \ballMink[D](z,r)$: 
\begin{enumerate}
\item[$(i)$] $\volMink[D](C_+ \cap B) ~ \CGEQ ~ \volMink[D](B)$
\item[$(ii)$] $\areaMink[D](\bd C_+ \cap \interior(B)) ~ \CGEQ ~ \areaMink[D](B)$.
\end{enumerate}
\end{lemma}

\begin{proof}
Fix $z \in \bd C_+$, $0 < r \leq \alpha$, and let $B := \ballMink[D](z,r)$. For part~(i), note that since 
\[
    C_+
        ~ = ~ \{p \in \RE^d \ST \distMink[D](p,C) \leq \alpha\},
\]
and $z \in \bd C_+$, we have $\distMink[D](z,C) = \alpha$. Let $x$ be any point in $C$ such that $\distMink[D](z,x) = \alpha$, and let $y$ be the point on the segment $\overline{z x}$ at distance $r/2$ from $z$ (see Figure~\ref{f:mink-exp-fat}(a)). We assert that 
\[
    \ballMink[D] \!\left(y,\frac{r}{2}\right)
        ~ \subseteq ~ C_+ \cap B,
\]

\begin{figure}[htbp]
  \centerline{\includegraphics[scale=0.40]{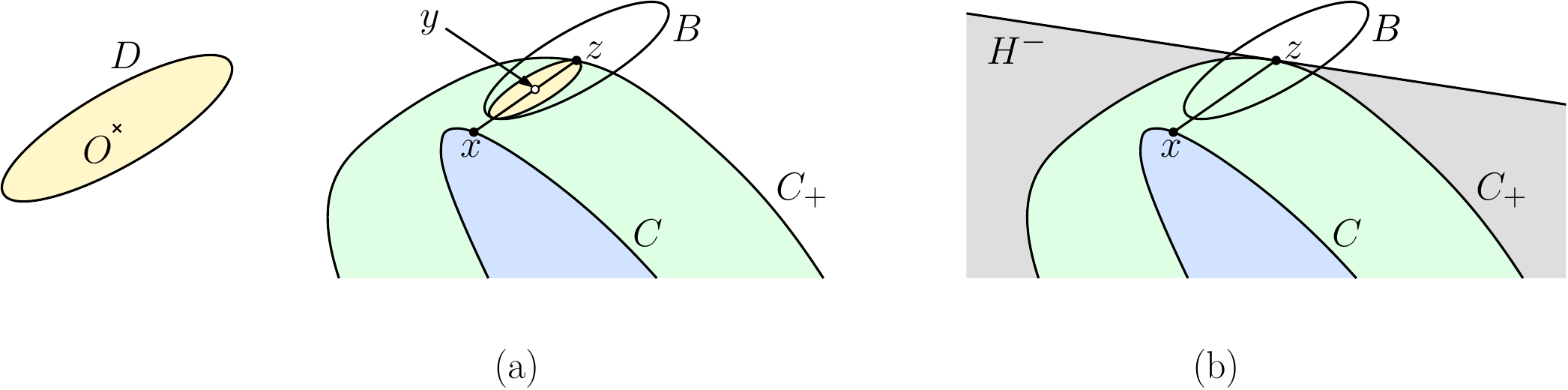}}
  \caption{Proof of Lemma~\ref{lem:mink-exp-fat}.}
  \label{f:mink-exp-fat}
\end{figure}

To see why, observe that by the triangle inequality any $p \in \ballMink[D](y,r/2)$ satisfies
\[
    \distMink[D](p,z)
        ~ \leq ~ \distMink[D](p,y) + \distMink[D](y,z)
        ~ \leq ~ r,
\]
implying $p \in B$. Similarly, 
\begin{align*}
    \distMink[D](p,x)
        & ~ \leq ~ \distMink[D](p,y) + \distMink[D](y,x)
          ~ =    ~ \distMink[D](p,y) + \bigl( \distMink[D](z,x) - \distMink[D](z,y) \bigr) \\
        & ~ \leq ~ \frac{r}{2} + \left( \alpha - \frac{r}{2} \right)
          ~ =    ~ \alpha,
\end{align*}
so $p \in \ballMink[D](x,\alpha) \subseteq C_+$, establishing the assertion. Therefore, we have
\[
    \volMink[D](C_+ \cap B)
        ~ \geq  ~ \volMink[D] \!\left( \ballMink[D] \left( y, \frac{r}{2} \right) \right)
        ~ =     ~ 2^{-d} \cdot \volMink[D](B)
        ~ \CGEQ ~ \volMink[D](B),
\]
which establishes~(i).

For part~(ii), we will employ the relative isoperimetric inequality for Minkowski balls. Let
\[
    \mu
        ~ := ~ \frac{\volMink[D](C_+ \cap B)}{\volMink[D](B)}.
\]
By part~(i), we have $\mu \CGEQ 1$. Let $H$ be a supporting hyperplane to $C_+$ at $z$, and let $H^-$ be the closed halfspace bounded by $H$ that contains $C_+$ (see Figure~\ref{f:mink-exp-fat}(b)). Then
\[
    C_+ \cap B 
        ~ \subseteq ~ B \cap H^-.
\]
Since $B$ is centrally symmetric about $z$ and $H$ passes through $z$, the hyperplane $H$ bisects $B$, and hence $\mu \leq 1/2$. Applying Lemma~\ref{lem:mink-HT-iso} to the convex body $C_+$ and the ball $B$, we obtain
\[
    \frac{\areaMink[D](\bd C_+ \cap \interior(B))}{\areaMink[D](B)}
        ~ \CGEQ ~ \min(\mu, 1-\mu)
        ~ =     ~ \mu
        ~ \CGEQ ~ 1,
\]
which proves~(ii).
\end{proof}

Next, we relate the $\alpha$-boundary covering number to the surface areas of $C$ and $C_+$.

\begin{lemma}\label{lem:mink-exp-surf}
Let $C$, $D$, $\alpha$, and $C_+$ be as in Lemma~\ref{lem:mink-exp-bdry}. Then
\[
    \frac{\areaMink[D](C)}{\omega_{d-1} \, \alpha^{d-1}}
        ~ \CLEQ ~ \NbdMink[D](C,\alpha)
        ~ \CLEQ ~ \frac{\areaMink[D](C_+)}{\omega_{d-1} \, \alpha^{d-1}}.
\]
\end{lemma}

\begin{proof}
We begin with the upper bound. Let $S \subset \bd C_+$ be a maximal $\alpha$-separated subset. Then the set of balls $\{\ballMink[D](s,\alpha)\}_{s \in S}$ covers $\bd C_+$, while $\{\ballMink[D](s,\frac{\alpha}{2})\}_{s \in S}$ have pairwise disjoint interiors. 

We begin by showing that $\{\ballMink[D](s, 2 \alpha)\}_{s \in S}$ covers $\bd C$. To prove this, consider any $x \in \bd C$. By Lemma~\ref{lem:mink-exp-bdry}, there exists $y \in \bd C_+$ such that $\distMink[D](x,y) = \alpha$. Since the $\alpha$-radius balls cover $\bd C_+$, there exists $s \in S$ such that $\distMink[D](s, y) \leq \alpha$. Hence, by the triangle inequality,
\[
    \distMink[D](s,x)
        ~ \leq ~ \distMink[D](s,y) + \distMink[D](y,x)
        ~ \leq ~ 2 \alpha,
\]
implying that $\{\ballMink[D](s, 2 \alpha)\}_{s \in S}$ covers $\bd C$, and hence, $\NbdMink[D](C, 2\alpha) \leq |S|$. By applying Lemma~\ref{lem:mink-ball-cover} (with $E = \bd C$), this implies that
\begin{equation} \label{eq:mink-exp-surf-1}
    \NbdMink[D](C, \alpha)
        ~ \CLEQ ~ \NbdMink[D](C, 2 \alpha)
        ~ \leq  ~ |S|.
\end{equation}

Appealing to the relative surface fatness of the expanded body, for each $s \in S$, we apply Lemma~\ref{lem:mink-exp-fat}(ii) (with $z = s$ and $r = \alpha/2$) and then apply the surface area bound of Lemma~\ref{lem:mink-ball-growth}(ii) to obtain
\begin{align*}
    \areaMink[D] \!\left( \bd C_+ \cap \interior \!\left( \ballMink[D] \left( s, \frac{\alpha}{2} \right) \right) \right)
        & ~ \CGEQ ~ \areaMink[D] \!\left( \ballMink[D] \left( s, \frac{\alpha}{2} \right) \right) \\
        & ~ \CGEQ ~ \omega_{d-1} \left( \frac{\alpha}{2} \right)^{d-1}.
\end{align*}
Summing over $S$ and by the disjointness of these interiors, we have
\[
    \areaMink[D](C_+)
        ~ \geq  ~ \sum_{s \in S} \areaMink[D] \!\left( \bd C_+ \cap \interior \!\left( \ballMink[D] \left( s, \frac{\alpha}{2} \right) \right) \right)
        ~ \CGEQ ~ |S| \, \omega_{d-1} \, \alpha^{d-1}.
\]
Combining this with Eq.~\eqref{eq:mink-exp-surf-1}, we conclude that
\[
    \NbdMink[D](C, \alpha)
        ~ \CLEQ ~ |S|
        ~ \CLEQ ~ \frac{\areaMink[D](C_+)}{\omega_{d-1}\alpha^{d-1}},
\]
which establishes the upper bound.

For the lower bound, let $S$ denote the set of centers of any minimum-sized covering of $\bd C$ by Minkowski balls of radius $\alpha$. By subadditivity,
\begin{equation} \label{eq:mink-exp-surf-2}
    \areaMink[D](C)
        ~ \leq ~ \sum_{s \in S} \areaMink[D] \bigl( \bd C \cap \ballMink[D](s, \alpha) \bigr).
\end{equation}
Also, for any $s \in S$,
\[
    \bd C \cap \ballMink[D](s,\alpha)
        ~ \subseteq ~ \bd \bigl( C \cap \ballMink[D](s, \alpha) \bigr),  
\]
and hence by simple inclusion and the monotonicity of Holmes--Thompson area for convex bodies (Lemma~\ref{lem:HT-area-monotone}), we have
\begin{equation} \label{eq:mink-exp-surf-3}
    \areaMink[D] \bigl( \bd C \cap \ballMink[D](s, \alpha) \bigr)
        ~ \leq ~ \areaMink[D] \bigl( C \cap \ballMink[D](s, \alpha) \bigr)
        ~ \leq ~ \areaMink[D] \bigl( \ballMink[D](s, \alpha) \bigr).
\end{equation}
Combining Eqs.~\eqref{eq:mink-exp-surf-2} and~\eqref{eq:mink-exp-surf-3} with our bound on the areas of Minkowski balls (Lemma~\ref{lem:mink-ball-growth}(ii)) yields
\[
    \areaMink[D](C)
            ~ \leq  ~ |S| \cdot \areaMink[D] \bigl( \ballMink[D](s, \alpha) \bigr)
            ~ \CLEQ ~ |S| \, \omega_{d-1} \, \alpha^{d-1},
\]
and hence
\[
    \NbdMink[D](C, \alpha)
        ~ =  ~  |S|
        ~ \CGEQ ~ \frac{\areaMink[D](C)}{\omega_{d-1} \, \alpha^{d-1}},
\]
which establishes the lower bound and completes the proof.
\end{proof}

\subsection{Boundary Properties in the Hilbert Setting} \label{s:prop-hilb-exp-surf}

We now establish the Hilbert analogs of the expansion properties used for boundary coverings. Let $K \subset \RE^d$ be a convex body, let $G \subset \interior(K)$ be a convex body, and let $0 < \alpha \leq 1$. Unlike the Minkowski case, the boundary-transfer property is more delicate, so we defer its proof to Appendix~\ref{app:hilb-exp-surf-i}. 

\begin{restatable}{lemma}{lemHilbExpBdry} \label{lem:hilb-exp-bdry}
Let $K$ be a convex body in $\RE^d$, let $G \subset \interior(K)$ be a convex body. For any $0 < \alpha \leq 1$, let $G_+ := \expandH[K](G,\alpha)$. Then for any $x \in \bd G$, there exists $y \in \bd G_+$ such that $\distHilb[K](x,y) = \alpha$. In particular,
\[
    \bd G
        ~ \subseteq ~ \{z \ST \distHilb[K](z,\bd G_+) \leq \alpha\}.
\]
\end{restatable}

Our next result shows that, at bounded scales, Hilbert expansions are 
relatively $\alpha$-fat. 

\begin{lemma}\label{lem:hilb-exp-fat}
Let $K$, $G$, $\alpha$, and $G_+$ be as in Lemma~\ref{lem:hilb-exp-bdry}. Then $G_+$ is relatively $\alpha$-fat and $\alpha$-surface-fat. More precisely, for any $z \in \bd G_+$, any $0 < r \leq \alpha$, and letting
$B := \ballHilb[K](z,r)$:
\begin{enumerate}
\item[$(i)$] $\volHilb[K](G_+ \cap B) ~ \CGEQ ~ \volHilb[K](B)$.
\item[$(ii)$] $\areaHilb[K](\bd G_+ \cap \interior(B)) ~ \CGEQ ~ \areaHilb[K](B)$.
\end{enumerate}
\end{lemma}

\begin{proof}
Fix $z \in \bd G_+$, $0 < r \leq \alpha$, and let $B := \ballHilb[K](z,r)$. For part~(i), the argument is the same inscribed-ball argument as in the Minkowski case. Since $z \in \bd G_+$, there exists $x \in G$ such that $\distHilb[K](z,x) = \alpha$. Let $y$ be the point on the segment $\overline{z x}$ at distance $r/2$ from $z$.

As in the Minkowski case (Lemma~\ref{lem:mink-exp-fat}), the triangle inequality implies that $\ballHilb[K]( y, r/2) \subseteq G_+ \cap B$, and hence,
\[
    \volHilb[K](G_+ \cap B)
        ~ \geq ~ \volHilb[K] \!\left( \ballHilb[K] \left( y, \frac{r}{2} \right) \right).
\]
Applying Lemma~\ref{lem:hilb-ball-growth}(i) with radii $r/2$ and $r$, and using $r \leq \alpha \leq 1$, we have
\[
    \volHilb[K](G_+ \cap B) ~ \CGEQ ~ \volHilb[K](B),
\]
which establishes~(i).

For part~(ii), we will employ the relative isoperimetric inequality for Hilbert balls. Let
\[
    \mu ~ := ~ \frac{\volHilb[K](G_+ \cap B)}{\volHilb[K](B)}.
\]
By part~(i), we have $\mu \CGEQ 1$. Applying Lemma~\ref{lem:hilbert-HT-iso} to the convex body $G_+$ and the ball $B$ yields
\[
    \areaHilb[K](\bd G_+ \cap \interior(B))
        ~ \CGEQ ~ \mu \cdot \areaHilb[K](B)
        ~ \CGEQ ~ \areaHilb[K](B),
\]
which proves~(ii).
\end{proof}

Next, we relate the $\alpha$-boundary covering number to the surface areas of $G$ and $G_+$.

\begin{lemma}\label{lem:hilb-exp-surf}
Let $K$, $G$, $\alpha$, and $G_+$ be as in Lemma~\ref{lem:hilb-exp-bdry}. Then
\[
    \frac{\areaHilb[K](G)}{\omega_{d-1} \, \alpha^{d-1}}
        ~ \CLEQ ~ \Nbd[K](G,\alpha)
        ~ \CLEQ ~ \frac{\areaHilb[K](G_+)}{\omega_{d-1} \, \alpha^{d-1}}.
\]
\end{lemma}

\begin{proof}
The proof is parallel to the Minkowski case (Lemma~\ref{lem:mink-exp-surf}) but by replacing the Minkowski-based lemmas with their Hilbert counterparts. In particular, for the upper bound, we use Lemma~\ref{lem:hilb-exp-bdry} in place of Lemma~\ref{lem:mink-exp-bdry}, Lemma~\ref{lem:hilb-exp-fat}(ii) in place of Lemma~\ref{lem:mink-exp-fat}(ii), and Lemmas~\ref{lem:hilb-ball-cover} and~\ref{lem:hilb-ball-growth}(ii) in place of the corresponding Minkowski ball estimates. 

For the lower bound, we argue as in the Minkowski case, using the monotonicity of the Hilbert Holmes--Thompson surface area (Lemma~\ref{lem:HT-area-monotone}) and Lemma~\ref{lem:hilb-ball-growth}(ii). Note that the assumption $\alpha \leq \RLocal/4$ ensures that all Hilbert radii involved remain in the bounded regime.
\end{proof}

\section{Duality for Boundary Coverings} \label{s:boundary-cover}

In this section, we prove our duality results for boundary coverings. We begin with a proof of the second part of Theorem~\ref{thm:hilbert-cover-duality}, which we restate below.

\thmHilbertCoverDuality*

\begin{proof} (Of part~(ii).)
Fix $\alpha \in (0, 1]$ and set $G_+ = \expandH[K](G,\alpha)$. By Lemma~\ref{lem:hilb-exp-i}, $G_+$ is a convex body with $G \subset \interior(G_+) \subset \interior(K)$. By the upper-bound part of Lemma~\ref{lem:hilb-exp-surf} and the Holmes--Thompson polarity identity (Lemma~\ref{lem:HT-hilb-polarity}(ii))
\begin{equation}\label{eq:boundary-cover-duality-hilb-1}
    \Nbd[K](G, \alpha)
        ~ \CLEQ ~ \frac{\areaHilb[K](G_+)}{\omega_{d-1} \, \alpha^{d-1}}
        ~ \CEQ    ~ \frac{\areaHilb[\polar{G_+}](\polar{K})}{\omega_{d-1} \, \alpha^{d-1}}.
\end{equation}
Applying the lower-bound part of Lemma~\ref{lem:hilb-exp-surf} in the Hilbert geometry induced by $\polar{G_+}$, gives
\begin{equation}\label{eq:boundary-cover-duality-hilb-2}
    \frac{\areaHilb[\polar{G_+}](\polar{K})}{\omega_{d-1} \, \alpha^{d-1}}
        ~ \CLEQ ~ \Nbd[\polar{G_+}](\polar{K}, \alpha).
\end{equation}
Since $\alpha \leq 1$ and Section~\ref{s:covering} fixes $\RLocal \geq 8$, we have $\alpha \leq \RLocal/8$. Hence Lemma~\ref{lem:polar-exp-cover-hilb} applies with $U = \bd \polar{K}$, yielding 
\begin{equation}\label{eq:boundary-cover-duality-hilb-3}
    \Nbd[\polar{G_+}](\polar{K}, \alpha)
        ~ \CLEQ ~ \Nbd[\polar{G}](\polar{K}, \alpha).
\end{equation} 
By chaining Eqs.~\eqref{eq:boundary-cover-duality-hilb-1}--\eqref{eq:boundary-cover-duality-hilb-3} together, we obtain
\[
    \Nbd[K](G,\alpha)
        ~ \CLEQ ~ \Nbd[\polar{G}](\polar{K}, \alpha),
\]
which establishes the second inequality. To establish the first inequality, apply the same argument with $(G, K)$ replaced by $(\polar{K}, \polar{G})$.
\end{proof}

\bigskip

We now consider the duality results for covering by translates of a centered convex body. For any convex body $C$ in $\RE^d$ with $O \in C$, define its \emph{symmetric core} and \emph{symmetric union}, respectively by
\[
    C_{\cap} ~ := ~ C \cap (-C)
    \quad\text{and}\quad
    C_{\cup} ~ := ~ \conv(C \cup -C).
\]
Thus, $C_{\cap}$ is the largest centrally symmetric convex body contained in $C$, and $C_{\cup}$ is the smallest centrally symmetric convex body containing $C$. It is well known that $\polar{C_{\cap}} = (\polar C)_{\cup}$ (see, e.g., \cite[Section~1.6]{Sch14}).

We will use the following two technical lemmas. The auxiliary lemmas underlying these reductions are routine adaptations of standard arguments, and their proofs are deferred to Appendix~\ref{app:asym-boundary-mink}.

\begin{restatable}{lemma}{lemSymUnionHT} \label{lem:sym-union-HT}
Let $D\subset \RE^d$ be a centrally symmetric convex body, and let $C\subset \RE^d$ be a convex body with $O\in C$. Then:
\begin{enumerate}
\item[$(i)$] $\volMink[D](C_{\cup}) \, \CLEQ \, \volMink[D](C)$ and
\item[$(ii)$] $\areaMink[D](C_{\cup}) \, \CLEQ \, \areaMink[D](C)$.
\end{enumerate}
\end{restatable}

\begin{restatable}[Covering with the symmetric core]{lemma}{lemCoreCover} \label{lem:core-cover}
Let $D$ be any convex body in $\RE^d$ such that the origin is either the centroid or the Santal\'o point of $D$. Then, for any set $U \subset \RE^d$ and any $\alpha > 0$,
\[
    \NcovMink[D_{\cap}](U, \alpha)
        ~ \CLEQ ~ \NcovMink[D](U, \alpha).
\]
\end{restatable}

Given these, we present our proof of Theorem~\ref{thm:mink-asym-boundary-cover-duality}, which is restated below. The proof proceeds in several steps. We begin by reducing the covering problem to a Holmes--Thompson area estimate in the geometry of the symmetric core. We then expand and symmetrize the body to be covered, allowing us to apply the duality results for symmetric bodies. Finally, the auxiliary symmetrizations are removed.

\thmMinkAsymBoundaryCoverDuality*

\begin{proof}
Define
\[
    D_{\cap} := D\cap(-D),\qquad
    C_{\cup} := \conv(C\cup -C),\qquad
    C_+ := C+\alpha D_{\cap},\qquad
    C_{\cup,+} := C_{\cup}+\alpha D_{\cap}.
\]

Since $D_{\cap} \subseteq D$, every translate of $\alpha D_{\cap}$ is contained in the corresponding translate of $\alpha D$. Therefore,
\begin{equation}\label{eq:bd-core-monotone}
    \NbdMink[D](C,\alpha)
        ~ \leq ~ \NbdMink[D_{\cap}](C,\alpha).
\end{equation}
By the upper-bound part of Lemma~\ref{lem:mink-exp-surf} in the Minkowski norm with unit ball $D_{\cap}$,
\begin{equation}\label{eq:main-step-bd-to-area}
    \NbdMink[D_{\cap}](C, \alpha)
        ~ \CLEQ ~ \frac{\areaMink[D_{\cap}](C_+)}{\omega_{d-1}\alpha^{d-1}}.
\end{equation}
Next, we pass to a symmetric intermediary on the primal side. Since $C\subseteq C_{\cup}$, we have
\[
    C_+
        ~ =         ~ C + \alpha D_{\cap}
        ~ \subseteq ~ C_{\cup} + \alpha D_{\cap}
        ~ =         ~ C_{\cup,+},
\]
and hence, by the monotonicity of Minkowski Holmes--Thompson area (Lemma~\ref{lem:HT-area-monotone}),
\begin{equation}\label{eq:main-area-enlarge}
    \areaMink[D_{\cap}](C_+)
        ~ \leq ~ \areaMink[D_{\cap}](C_{\cup,+}).
\end{equation}
Since both $D_{\cap}$ and $C_{\cup,+}$ are centrally symmetric, Holmes--Thompson duality 
(Lemma~\ref{lem:HT-mink-polarity}(ii)) implies that
\begin{equation}\label{eq:main-area-dualize}
    \areaMink[D_{\cap}](C_{\cup,+})
        ~ = ~ \areaMink[\polar{C_{\cup,+}}](\polar{D_{\cap}}).
\end{equation}
Recalling that $\polar{D_{\cap}} = (\polar D)_{\cup}$ and applying Lemma~\ref{lem:sym-union-HT}(ii) with $D := \polar{C_{\cup,+}}$ and $C := \polar D$, implies
\begin{equation}\label{eq:main-step3-output}
    \areaMink[\polar{C_{\cup,+}}](\polar{D_{\cap}})
        ~ \CLEQ ~ \areaMink[\polar{C_{\cup,+}}](\polar D).
\end{equation}
Now, applying the lower-bound part of Lemma~\ref{lem:mink-exp-surf} in the Minkowski norm with unit ball $D := \polar{C_{\cup,+}}$ and $C := \polar D$ yields 
\begin{equation}\label{eq:main-step3-area-to-cover}
    \frac{\areaMink[\polar{C_{\cup,+}}](\polar D)}{\omega_{d-1}\alpha^{d-1}}
        ~ \CLEQ ~ \NbdMink[\polar{C_{\cup,+}}](\polar D,\alpha).
\end{equation}
Combining Eqs.~\eqref{eq:bd-core-monotone}--\eqref{eq:main-step3-area-to-cover}, we have
\begin{equation}\label{eq:main-step4-output}
    \NbdMink[D](C, \alpha)
        ~ \CLEQ ~ \NbdMink[\polar{C_{\cup,+}}](\polar D,\alpha).
\end{equation}

It remains to remove the expansion ($+$) and the auxiliary symmetrization. Since $C_{\cup,+} = C_{\cup} + \alpha D_{\cap}$, Lemma~\ref{lem:polar-exp-cover-mink} applied with $C := C_{\cup}$, $D := D_{\cap}$, and $U := \bd \polar D$ gives
\begin{equation}\label{eq:main-step4-remove-plus}
    \NbdMink[\polar{C_{\cup,+}}](\polar D, \alpha)
        ~ \CLEQ ~ \NbdMink[\polar{C_{\cup}}](\polar D, \alpha).
\end{equation}
Again recall that $\polar{C_{\cup}} = (\polar C)_{\cap}$. Because $C$ is centered, the Santal\'o point of $\polar C$ is the origin. We can therefore apply Lemma~\ref{lem:core-cover} with $D := \polar C$ and $U := \bd \polar{D}$, to obtain
\begin{equation}\label{eq:main-step5-core}
    \NbdMink[\polar{C_{\cup}}](\polar D,\alpha)
        ~ \CLEQ ~ \NbdMink[\polar C](\polar D,\alpha).
\end{equation}
Now, combining Eqs.~\eqref{eq:main-step4-output}--\eqref{eq:main-step5-core}, we conclude that
\[
    \NbdMink[D](C, \alpha)
        ~ \CLEQ ~ \NbdMink[\polar C](\polar D, \alpha),
\]
which establishes the ``$\CLEQ$'' part of the result.

To complete the proof, we apply the same argument but with $(C, D)$ replaced by $(\polar D, \polar C)$. By
bipolarity, we obtain the reverse inequality
\[
    \NbdMink[\polar C](\polar D, \alpha)
        ~ \CLEQ ~ \NbdMink[D](C,\alpha).
\]
Here, the final application of Lemma~\ref{lem:core-cover} is to the body $D$, and its centering hypothesis holds because $D$ is centered. In conclusion, we have
\[
    \NbdMink[D](C, \alpha)
        ~ \CEQ ~ \NbdMink[\polar C](\polar D,\alpha),
\]
as desired.
\end{proof}

\section{Concluding Remarks} \label{s:conclusion}

We have established new polarity-duality results in two contexts. First, we considered both volumetric and boundary coverings in Hilbert geometries. We showed that for convex bodies $G$ and $K$ in $\RE^d$, where $O \in \interior(G)$ and $G \subset \interior(K)$ with radii $\alpha \in (0,1]$, the covering problem for $G$ using $\alpha$-balls in the Hilbert geometry induced by $K$ is equivalent, up to a universal loss $c^{d}$, to the dual covering problem for $\polar{K}$ in the Hilbert geometry induced by $\polar{G}$. Second, we considered both volumetric and boundary coverings by scaled translates of convex bodies, assuming that both bodies are centered.

Our proofs reveal a common mechanism across all the duality statements in the paper. First, one passes to an $\alpha$-expansion, converts covering estimates into Holmes--Thompson volume or area bounds, dualizes these bounds, and then transfers the resulting estimate back to the target geometry. A basic difficulty is that, after dualization, the natural ambient geometry is induced by $\polar{G_+}$ rather than by $\polar{G}$. To address this, we presented a polarity-expansion stability principle that controls this change of ambient geometry on the common domain $\polar{K}$.

For boundary coverings, it is necessary to overcome a second difficulty that is absent in the volumetric case: a boundary-centered metric ball may capture arbitrarily little boundary area unless one first enforces scale regularity. Here, the expansion has a genuine regularizing effect. Combined with bounded-radius localization and a localized relative isoperimetric inequality, this yields the surface-area bounds needed to carry out the same duality argument in the boundary setting.

Several questions remain. The first concerns the bounded-radius restriction in the Hilbert case. Since our arguments are essentially local, they apply only in the small-radius regime $\alpha\in(0,1]$. It would be interesting to know whether a comparable polarity duality continues to hold for all $\alpha > 0$, or whether the global Hilbert geometry of the ambient body creates a genuine obstruction at larger scales.

One may also formulate Pietsch-type duality conjectures for each of the covering duality problems considered here. More precisely, one may ask whether the present $c^{d}$-loss can be replaced by a dimension-free duality estimate in the sense of metric entropy, perhaps after an absolute rescaling of the radius parameter.

More broadly, the present results suggest the possibility of an entropy theory in Hilbert geometry, with Hilbert balls serving as intrinsic counterparts of Minkowski balls. This, in turn, raises the question of which aspects of the present duality mechanism extend beyond Hilbert geometry to other convex-projective or Finsler settings.


\pdfbookmark[1]{References}{s:ref}
\bibliographystyle{plainurl}
\bibliography{shortcuts,convex}


\appendix

\section{Proofs of Covering Lemmas} \label{app:covering-lemmas}

In this appendix, we prove the covering lemmas from Section~\ref{s:covering}. The Hilbert proofs use the bounded-radius Holmes--Thompson volume growth estimate established in Lemma~\ref{lem:hilb-ball-growth}.

\lemMinkBallCover*

\begin{proof}
Let $n = \NcovMink[D](U, r')$, and fix an $r'$-cover of $U$ by translates $x_i + r' D$, for $i = 1,\dots, n$. It suffices to show that each translate of $r' D$ is coverable by $\CLEQ (r'/r)^d$ translates of $r D$.

By translation and scaling, it suffices to bound $\NcovMink[D](r' D,r)$. Let $S\subseteq r'D$ be a maximal $r$-separated set in the Minkowski metric induced by $D$. Then the translates $\{s + r D \ST s \in S\}$ cover $r' D$, while the sets $\{s + (r/2) D \ST s \in S\}$ are pairwise disjoint. Since $S \subseteq r' D$ and $r \leq r'$, each set $s + (r/2) D$ is contained in $(r' + r/2) D \subseteq (3 r'/2)D$. Comparing Euclidean volumes gives
\[
    |S| \left( \frac{r}{2} \right)^d \vol(D)
        ~ \leq ~ \left( \frac{3 r'}{2} \right)^d \vol(D),
\]
hence $|S| \CLEQ (r'/r)^d$. Refining each of the $n$ sets in this way yields an $r$-covering of $U$ of cardinality $\CLEQ (r'/r)^d n$.
\end{proof}

\lemHilbBallCover*

\begin{proof}
The proof is identical to that of Lemma~\ref{lem:mink-ball-cover}, with Hilbert balls in place of translates of $r D$ and using Lemma~\ref{lem:hilb-ball-growth}(i) in place of the Euclidean volume comparison. Indeed, a maximal $r$-separated subset of a ball $\ballHilb[K](x, r')$ yields a cover by radius-$r$ balls and a disjoint family of radius-$r/2$ balls contained in $\ballHilb[K](x, 3r'/2)$. The assumption $r' \leq \RLocal/2$ ensures that Lemma~\ref{lem:hilb-ball-growth}(i) applies.
\end{proof}

\section{Proofs of Relative Isoperimetric Inequalities} \label{app:rel-iso}

This appendix establishes the localized relative isoperimetric inequalities used in the boundary-covering arguments, both in the Minkowski case (Lemma~\ref{lem:mink-HT-iso}) and in the Hilbert case (Lemma~\ref{lem:hilbert-HT-iso}). The former is obtained by combining a Euclidean relative perimeter estimate for isotropic convex bodies with a transfer to Busemann and then Holmes--Thompson normalization. The latter is proved by reducing the bounded-radius Hilbert geometry to a comparable local Minkowski model via the Macbeath-region sandwich (Lemma~\ref{lem:mac-hilbert}) and the local comparison of Holmes--Thompson densities (Lemma~\ref{lem:local-ht-compare}). Throughout, the key scale-invariant quantity is the ratio $\beta/\mu$, where $\mu$ denotes the relevant volume fraction and $\beta$ the corresponding internal boundary-area fraction.

We first consider the Euclidean case. We say that a convex body $B$ in $\RE^d$ is in \emph{isotropic position}, or simply \emph{isotropic}, if the uniform probability measure on $B$ satisfies
\begin{equation} \label{eq:isotropic-assumptions}
    \int_B x \, dx ~ = ~ 0 
    \qquad\text{and}\qquad
    \frac{1}{\lambda_d(B)} \int_B \inner{x}{u}^2 \, dx ~ = ~ 1, \text{~~for all $u \in S^{d-1}$}. 
\end{equation}

\begin{lemma}[Euclidean relative isoperimetry] \label{lem:euclid-beta-gamma}
Let $B$ be any centrally symmetric, isotropic convex body in $\RE^d$, and let $E \subset \RE^d$ be any convex body such that
\[
    0 ~ < ~ \lambda_d \bigl( E \cap B \bigr) ~ < ~ \lambda_d(B).
\]
Set
\[
    \mu ~ := ~ \frac{\lambda_d(E \cap B)}{\lambda_d(B)} 
        \qquad\text{and}\qquad
    \beta ~ := ~ \frac{\lambda_{d-1}(\bd E\cap \interior(B))}{\lambda_{d-1}(\bd B)}.
\]
Then there exists an absolute constant $\ce > 0$ such that
\[
    \beta 
        ~ \geq ~ \frac{\ce}{d\sqrt{\log d}} \cdot \min\{\mu, 1-\mu\}.
\]
\end{lemma}

\begin{proof}
Define the uniform probability measure $Q_B$ supported on $B$ and its associated density $q_B$ by
\[
    Q_B(U) := \frac{\lambda_d(U \cap B)}{\lambda_d(B)}
    \quad\text{and}\quad
    q_B(x) := \lambda_d(B)^{-1}\mathbf{1}_B(x),
\]
for all Lebesgue measurable $U \subset \RE^d$. Let $\chi(Q_B)$ denote the Cheeger isoperimetric constant of $Q_B$, that is,
\[
    \chi(Q_B)
        ~ := ~ \inf_{U} \frac{\int_{\bd U} q_B\, d\lambda_{d-1}}{\min\{Q_B(U),\, 1 - Q_B(U)\}},
\]
where the infimum is over all open sets $U \subset \RE^d$ with smooth boundary and $0 < Q_B(U) < 1$. By Klartag's bound for isotropic log-concave measures (see, e.g., \cite[Theorem~1.2]{Kla23}),
\begin{equation} \label{eq:cheeger-klartag}
    \chi(Q_B) ~ \geq ~ \frac{\ce}{\sqrt{\log d}},
\end{equation}
for an absolute constant $\ce > 0$. By approximation, equivalently by the finite-perimeter formulation for log-concave measures, the defining inequality for $\chi(Q_B)$ also applies to convex bodies such as $E$. Since $Q_B(E)=\mu$ and
\[
    \int_{\bd E} q_B\, d\lambda_{d-1}
        ~ = ~ \frac{1}{\lambda_d(B)}\,\lambda_{d-1}(\bd E \cap \interior(B)),
\]
it follows that
\begin{equation} \label{eq:area-lowerbound}
    \lambda_{d-1}(\bd E \cap \interior(B))
        ~ \geq ~ \chi(Q_B)\,\lambda_d(B)\cdot \min\{\mu,1-\mu\}.
\end{equation}

It remains to estimate $\lambda_d(B)/\lambda_{d-1}(\bd B)$. Fix any $u\in S^{d-1}$. Consider a random variable $Y \sim Q_B$. Since $B$ is centrally symmetric, $\inner{Y}{u}$ is supported in $[-h_B(u),h_B(u)]$, and satisfies $\Var\inner{Y}{u} = 1$ by isotropy. It follows that $h_B(u) \geq 1$ for all $u \in S^{d-1}$. Recall that $\norm[B](x)$ denotes the outer unit normal at a smooth point $x\in\bd B$, and is defined at all boundary points except on a set of Hausdorff measure zero. At all such points, we have
\[
    \inner{x}{\norm[B](x)} 
        ~ =    ~ h_B(\norm[B](x)) 
        ~ \geq ~ 1.
\]
Therefore, by the divergence theorem,
\[
    \lambda_d(B)
        ~ =    ~ \frac{1}{d}\int_{\bd B}\inner{x}{\norm[B](x)}\,d\lambda_{d-1}(x)
        ~ \geq ~ \frac{1}{d}\lambda_{d-1}(\bd B),
\]
and hence
\begin{equation} \label{eq:vol-area-ratio}
    \frac{\lambda_d(B)}{\lambda_{d-1}(\bd B)}\ \ge\ \frac{1}{d}.
\end{equation}
Combining Eqs.~\eqref{eq:cheeger-klartag}--\eqref{eq:vol-area-ratio}, we have
\begin{align*}
    \beta
        & ~ = ~ \frac{\lambda_{d-1}(\bd E\cap \interior(B))}{\lambda_{d-1}(\bd B)}
          ~ \geq ~ \frac{\chi(Q_B)\,\lambda_d(B) \cdot \min\{\mu,1-\mu\}}{\lambda_{d-1}(\bd B)} \\
        & ~ \geq ~ \frac{\ce}{\sqrt{\log d}} \left( \frac{\lambda_d(B)}{\lambda_{d-1}(\bd B)} \right) \min\{\mu,1-\mu\}
          ~ \geq ~ \frac{\ce}{d \sqrt{\log d}} \cdot \min\{\mu,1-\mu\},
\end{align*}
as desired.
\end{proof}

\subsection{Relative Isoperimetry in a Minkowski Ball} \label{s:mink-reliso}

In this section, we work with the Busemann volume and surface area induced by the Minkowski norm with a centrally symmetric convex unit ball $D$, denoted $\volBus[D](\cdot)$ and $\areaBus[D](\cdot)$, respectively. These notions are used first to prove the relative isoperimetric inequality of Lemma~\ref{lem:busemann-iso}. We then transfer this bound to the Holmes--Thompson setting and, by a final translation and scaling normalization, obtain Lemma~\ref{lem:mink-HT-iso}. 

Let $D \subset \RE^d$ be a centrally symmetric convex body. We normalize the Busemann volume so that $\volBus[D](D) = \omega_d$. Generally, for every Lebesgue measurable set $U \subset \RE^d$, 
\begin{equation} \label{eq:Bus-volume-density}
    \volBus[D](U)
        ~ := ~ \frac{\omega_d}{\lambda_d(D)} \, \lambda_d(U).
\end{equation}

For a piecewise $C^1$ surface $S$ and $x \in S$, let $T_x$ denote the tangent space at $x$, identified with the corresponding $(d-1)$-dimensional subspace of $\RE^d$. The \emph{Busemann $(d-1)$-area element} is defined by
\begin{equation} \label{eq:Bus-area}
    d\areaBus[D](x)
        ~ := ~ \frac{\omega_{d-1}}{\lambda_{d-1}(D \cap T_x)}\, d\lambda_{d-1}(x).
\end{equation}
When $G$ is a convex body, $\areaBus[D](G)$ denotes $\areaBus[D](\bd G)$.

The next lemma presents the linear invariance of the Busemann volume and surface area under invertible linear transformations (see, e.g., \cite[Axioms~2.2(c)]{HoT79} for the linear invariance of the Busemann volume, and \cite{Sch01a} for the associated Busemann area functional). Recall that $\|\cdot\|_D$ denotes the Minkowski norm with unit ball $D$, as defined in Section~\ref{s:prelim}.

\begin{lemma} \label{lem:linear-invariance}
Given any centrally symmetric convex body $D$ in $\RE^d$ and any invertible linear transformation $L:\RE^d \to \RE^d$, for all $x \in \RE^d$, $\|L x\|_{L(D)} = \|x\|_D$. 
\end{lemma}

A direct consequence is that for any Lebesgue measurable set $U \subset \RE^d$ and any piecewise $C^1$ surface $S$,
\[
    \volBus[L(D)](L(U)) ~ = ~ \volBus[D](U)
    \quad\text{and}\quad
    \areaBus[L(D)](L(S)) ~ = ~ \areaBus[D](S).
\]
Next, we show that, relative to the Busemann measures, balls satisfy the relative isoperimetry.

\begin{lemma}[Busemann relative isoperimetry] \label{lem:busemann-iso}
Let $D$ be a centrally symmetric convex body in $\RE^d$, let $E \subset \RE^d$ be a convex body with $0 < \volBus[D](E \cap D) < \volBus[D](D)$, and let
\[
    \mu ~ := ~ \frac{\volBus[D](E \cap D)}{\volBus[D](D)} 
    \qquad\text{and}\qquad
    \beta ~ := ~ \frac{\areaBus[D](\bd E\cap \interior(D))}{\areaBus[D](D)}.
\]
Then there exists an absolute constant $\cb > 0$ such that
\[
    \beta ~ \geq ~ \frac{\cb}{d\sqrt{\log d}}\cdot \min\{\mu,1-\mu\}.
\]
\end{lemma}

Note that the $d^{-1}$ dependence in the previous result is optimal up to the $(\log d)^{-1/2}$ factor. Indeed, if $D = [-1, 1]^d$ and $E := \{x_1 \leq 0\}$, then $\mu = 1/2$, while 
\[
    \bd E \cap \interior(D) 
        ~ = ~ \{x_1 = 0\} \cap (-1, 1)^d
\]
is a central section parallel to the facets of $D$. Since the Busemann area density on a flat piece depends only on its normal direction, $\areaBus[D](\bd E \cap \interior(D))$ equals the Busemann area of any facet of $D$. As $\bd D$ consists of $2 d$ facets, it follows that
\[
    \frac{\areaBus[D](\bd E \cap \interior(D))}{\areaBus[D](D)}
        ~ = ~ \frac{1}{2d},
\]
which establishes the $d^{-1}$ dependence.

\smallskip

Before giving the proof, we need to precondition $D$ by transforming it into isotropic position. By Lemma~\ref{lem:linear-invariance}, we may apply a linear transformation that does this. Following this, the proof combines two ingredients: the Euclidean relative isoperimetric bound from the preceding subsection (Lemma~\ref{lem:euclid-beta-gamma}), and a comparison between the Busemann and Euclidean surface area densities, which we present next. 

\begin{lemma} \label{lem:area-compare}
There exists an absolute constant $\ccs \geq 1$ such that, for any centrally symmetric, isotropic convex body $D$ in $\RE^d$, and any piecewise $C^1$ surface $S$,
\[
    m_D \,\lambda_{d-1}(S)
        ~ \leq ~ \areaBus[D](S)
        ~ \leq ~ \ccs \, m_D \,\lambda_{d-1}(S),
\]
where,
\[
    m_D
        ~ := ~ \inf_H \frac{\omega_{d-1}}{\lambda_{d-1}(D \cap H)},
\]
and the infimum is over all $(d-1)$-dimensional linear subspaces $H$ in $\RE^d$.
\end{lemma}

\begin{proof} 
By definition,
\[
    \areaBus[D](S)
        ~ = ~ \int_S \frac{\omega_{d-1}}{\lambda_{d-1}(D \cap T_x)} \, d\lambda_{d-1}(x).
\]
Since $\omega_{d-1}/\lambda_{d-1}(D \cap T_x) \geq m_D$, for all $x \in S$ excluding a set of Hausdorff measure zero, the lower bound is immediate.

For the upper bound, it is standard that if $D$ is isotropic, then all central hyperplane sections of $D$ are comparable up to an absolute constant $\ccs$ (see, e.g., \cite[Section~3]{BGV14}). Hence
\[
    \frac{\omega_{d-1}}{\lambda_{d-1}(D\cap H)}
        ~ \leq ~ \ccs\, m_D
\]
for every $(d-1)$-dimensional linear subspace $H\subset \RE^d$. Applying this with $H = T_x$ and integrating over $S$ gives
\[
    \areaBus[D](S)
        ~ \leq ~ \ccs\, m_D\, \lambda_{d-1}(S).   \qedhere
\]
\end{proof}

\smallskip

\begin{proof}[Proof of Lemma~\ref{lem:busemann-iso}]
By Lemma~\ref{lem:linear-invariance}, we may assume that $D$ is in isotropic position. By Eq.~\eqref{eq:Bus-volume-density}, we have
\[
    \mu
        ~ = ~ \frac{\volBus[D](E \cap D)}{\volBus[D](D)} 
        ~ = ~ \frac{\lambda_d(E \cap D)}{\lambda_d(D)}. 
\]
Applying the lower bound of Lemma~\ref{lem:area-compare} to $S := \bd E \cap \interior(D)$ and the upper bound to $S := \bd D$, we obtain
\[
    \beta
        ~ =    ~ \frac{\areaBus[D](\bd E\cap\interior(D))}{\areaBus[D](D)}
        ~ \geq ~ \frac{1}{\ccs} \cdot \frac{\lambda_{d-1}(\bd E \cap \interior(D))}{\lambda_{d-1}(\bd D)}.
\]
By Lemma~\ref{lem:euclid-beta-gamma} applied to $D$ and $E$, we obtain
\[
    \frac{\lambda_{d-1}(\bd E\cap\interior(D))}{\lambda_{d-1}(\bd D)}
        ~ \geq ~ \frac{\ce}{d \sqrt{\log d}} \cdot \min\{\mu,1-\mu\}.
\]
Combining the last two inequalities and setting $\cb := \ce/\ccs$ completes the proof.
\end{proof}

We now convert the Busemann bound into the Holmes--Thompson (Minkowski) normalization used elsewhere in the paper. The following comparison is standard in the theory of volumes and areas in normed, more generally Finsler, spaces (see, e.g., \cite{AlT04} and the references therein). We include a short proof for completeness.

\begin{lemma}[Busemann vs.\ Holmes--Thompson (Minkowski)]\label{lem:BH-HT-compare-mink}
Let $D$ be a centrally symmetric convex body in $\RE^d$, viewed as the unit ball of the Minkowski norm. 
Then for any Lebesgue measurable set $U \subset \RE^d$ and any piecewise $C^1$ surface $S\subset \RE^d$,
\[
    \volMink[D](U) ~ \CEQ ~ \volBus[D](U),
    \qquad\text{and}\qquad
    \areaMink[D](S) ~ \CEQ ~ \areaBus[D](S).
\]
\end{lemma}

\begin{proof}
In a Minkowski space with unit ball $D$, both $\volBus[D]$ and $\volMink[D]$ are constant
multiples of the corresponding Lebesgue measure, or more precisely,
\[
    \volBus[D](U) ~ = ~ \frac{\omega_d}{\lambda_d(D)}\,\lambda_d(U),
    \quad\text{and}\quad
    \volMink[D](U) ~ = ~ \frac{\lambda_d(\polar{D})}{\omega_d}\,\lambda_d(U),
\]
and therefore
\[
    \frac{\volMink[D](U)}{\volBus[D](U)}
        ~ = ~ \frac{\lambda_d(D)\lambda_d(\polar{D})}{\omega_d^2}.
\]
The right-hand side is the normalized volume product of $D$. By the Blaschke--Santal\'o inequality together with a reverse Santal\'o inequality, 
\[
    \frac{\lambda_d(D)\lambda_d(\polar{D})}{\omega_d^2} ~ \CEQ ~ 1, 
\]
and the volume comparison follows. 

For the area comparison, fix any smooth point $x \in S$. Then $D \cap T_x$ is the unit ball of the induced norm on $T_x$ and the ratio of the Holmes--Thompson and Busemann $(d-1)$-area densities at $x$
equals the normalized volume product of $D \cap T_x$ in dimension $d-1$. Hence, this ratio is $\CEQ 1$, and integrating over $S$ gives $\areaMink[D](S) \CEQ \areaBus[D](S)$. 
\end{proof}

{\lemMinkHTiso*}

\begin{proof}
Let $T(y):=(y-z)/r$ and set $\widetilde{E}:=T(E)$. Then $T(B)=D$, and by translation invariance together with the homogeneity of Minkowski Holmes--Thompson volume and area in the fixed norm with unit ball $D$, the quantities $\mu$ and $\beta$ are unchanged when $(E, B)$ is replaced by $(\widetilde{E}, D)$. 

Thus, it suffices to treat the case $B=D$. Since $\volBus[D](\cdot)$ and $\volMink[D](\cdot)$ are both constant multiples of Lebesgue measure in the Minkowski space with unit ball $D$, their normalized volume fractions coincide, and hence 
\[
    \mu ~ = ~ \frac{\volBus[D]\bigl( \widetilde{E} \cap D \bigr)}{\volBus[D](D)}. 
\]
Writing
\[
    \beta_{\mathrm{Bus}} ~ := ~ \frac{\areaBus[D]\bigl(\bd \widetilde{E} \cap \interior(D)\bigr)}{\areaBus[D](D)},
\]
Lemma~\ref{lem:BH-HT-compare-mink} implies that $\beta \CGEQ \beta_{\mathrm{Bus}}$. Applying Lemma~\ref{lem:busemann-iso} to $(\widetilde{E},D)$ yields
\[
    \beta
        ~ \CGEQ ~ \beta_{\mathrm{Bus}}
        ~ \geq  ~ \frac{c_b}{d\sqrt{\log d}} \min\{\mu, 1 - \mu\}
        ~ \CGEQ ~ \min\{\mu, 1 - \mu\},
\]
where the last step absorbs the factor $\bigl( d \sqrt{\log d} \bigr)^{-1}$ into the base constant. 
\end{proof}

\subsection{From Minkowski to Hilbert Isoperimetry} \label{s:hilbert-reliso}

We now derive a relative isoperimetric inequality for Holmes--Thompson area in Hilbert geometry, localized to a Hilbert ball $B := \ballHilb[K](x,r)$ of radius $0 < r \leq 1$. The proof reduces Hilbert Holmes--Thompson isoperimetry on $B$ to Minkowski Holmes--Thompson isoperimetry on a suitable Macbeath region, via the local comparison estimate of Lemma~\ref{lem:local-ht-compare}.

For our applications, it suffices to establish a localized, anchored relative isoperimetric inequality in Hilbert geometry. In the boundary covering arguments, each application takes place inside $B$ and involves a convex cut $E \subseteq \interior(K)$ whose boundary passes through the center, that is, with $x \in \bd E$. This is precisely the setting of the next lemma. Although a more symmetric statement of the form $\beta \CGEQ \min\{\mu,1-\mu\}$ is natural in the Minkowski setting, we do not pursue an unanchored analog in the Hilbert case, since it is not needed for our applications and our argument relies essentially on the anchoring at $x$. 

\lemHilbertHTiso*

\begin{proof}
Let
\[
    A := M_K(x,1)-x, \qquad
    B_0 := M_K(x,\sigma r) = x + \sigma r\,A, \qquad
    B_1 := M_K(x,\tau r) = x + \tau r\,A,
\]
where $\sigma$ and $\tau$ are the constants from Lemma~\ref{lem:mac-hilbert}. Clearly, $B_0 \subseteq B \subseteq B_1$. Set $t := \tau/\sigma$, so that $B_1 = x + t(B_0 - x)$. Recall from Section~\ref{s:covering} that the global constant $\RLocal$ was chosen so that $\RLocal \geq 8$. Define
\[
    \widehat B ~ := ~ \ballHilb[K](x,\RLocal).
\]
Since $B=\ballHilb[K](x,r)$ with $0<r\leq 1$, we have $B \subseteq \widehat B$. Define
\[
    \mu_A ~ := ~ \frac{\volMink[A](E \cap B)}{\volMink[A](B)}
    \qquad\text{and}\qquad
    \mu_{A,0} ~ := ~ \frac{\volMink[A](E \cap B_0)}{\volMink[A](B_0)}.
\]
By Lemma~\ref{lem:local-ht-compare}(i), applied on $\widehat B$, and absorbing the resulting absolute constants into the notations $\CLEQ$, $\CGEQ$, and $\CEQ$, we have
\[
    \mu ~ \CEQ ~ \mu_A.
\]
Moreover, since $E \cap B_0 \subseteq E \cap B$ and $B_0 \subseteq B \subseteq B_1 = x + t(B_0 - x)$, translation invariance and homogeneity of $\volMink[A]$ implies that
\[
    \mu_{A,0} ~ \CGEQ ~ \mu_A.
\]
Let $H$ be a supporting hyperplane to $E$ at $x$, and let $H^{-}$ be the closed halfspace bounded by $H$ that contains $E$. Then $E \cap B_0 \subseteq B_0 \cap H^{-}$. Since $B_0$ is centrally symmetric about $x$ and $H$ passes through $x$, the hyperplane $H$ bisects $B_0$, and hence $\mu_{A,0}\leq 1/2$. Applying Lemma~\ref{lem:mink-HT-iso} to the Minkowski ball $B_0$, we obtain
\[
    \frac{\areaMink[A](\bd E \cap \interior(B_0))}{\areaMink[A](B_0)}
        ~ \CGEQ ~ \mu_{A,0}.
\]
Since $B_0 \subseteq B$,
\[
    \areaMink[A](\bd E \cap \interior(B))
        ~ \CGEQ ~ \areaMink[A](\bd E \cap \interior(B_0)).
\]
Also, from $B \subseteq B_1 = x + t(B_0 - x)$, together with monotonicity (Lemma~\ref{lem:HT-area-monotone}) and homogeneity, we have $\areaMink[A](B) \CLEQ \areaMink[A](B_0)$, and hence,
\[
    \frac{\areaMink[A](\bd E \cap \interior(B))}{\areaMink[A](B)}
        ~ \CGEQ ~ \mu_{A,0}.
\]
Finally, by applying Lemma~\ref{lem:local-ht-compare}(ii) on $\widehat B$, we have
\[
    \areaHilb[K](\bd E \cap \interior(B)) ~ \CEQ ~ \areaMink[A](\bd E \cap \interior(B))
    \qquad\text{and}\qquad
    \areaHilb[K](B) ~ \CEQ ~ \areaMink[A](B),
\]
and therefore
\[
    \beta
        ~ =     ~ \frac{\areaHilb[K](\bd E \cap \interior(B))}{\areaHilb[K](B)}
        ~ \CGEQ ~ \mu_{A,0}
        ~ \CGEQ ~ \mu_A
        ~ \CGEQ ~ \mu,
\]
as desired.
\end{proof}

\section{Duality of Funk Measures} \label{app:funk-polarity}

In this appendix, we present proofs of the Funk Holmes--Thompson polarity identities stated in Lemma~\ref{lem:funk-HT-polarity}. These identities were proved by Faifman~\cite{Fai24}. We include the simple volume proof for completeness, and we present our area proof as an elementary alternative to Faifman's proof. (Yet another proof based on a cone-based formulation of Funk volumes was presented in~\cite{ArM26arxiv}.)

We begin by presenting several useful definitions and technical results on the effects of linear and projective transformations on measures. Given an invertible linear operator $L$ on $\RE^d$, let $L^T$ denote its transpose and $\det(L)$ its determinant. The \emph{cofactor} of $L$ is defined by
\[
    \cofactor(L) ~ := ~ \det(L) \, L^{-T}.
\]
Given a map $f: \RE^d \to \RE^d$ that is differentiable at $x$, its derivative at $x$ is the linear transformation $Df(x)$ such that
\[
    f(x+h) ~ = ~ f(x) + Df(x)h + o(\|h\|),
    \quad\text{as $h \to 0$}.
\]

For an oriented $C^1$ hypersurface patch $S\subset\RE^d$ and $p \in S$, let $\norm[S](p)$ be a measurable unit normal on $S$ at point $p$. Its oriented area vector is
\[
    \mathcal A(S) ~ := ~ \int_S \norm[S](p)\,d\lambda_{d-1}(p).
\]

The next lemma states the standard transformation rules for volume and oriented area under an invertible linear transformation (see, for example, \cite[Section~3.2]{AnG20}).

\begin{lemma}\label{lem:volume-area-transform}
Let $L: \RE^d \to \RE^d$ be any invertible linear transformation. Then:
\begin{enumerate}
\item[$(i)$] For any Lebesgue measurable set $U\subset\RE^d$, $\lambda_d(L(U)) = |\det(L)| \, \lambda_d(U)$, and in particular, the volume element transforms as
\[
    d\lambda_d(Lu) ~ = ~ |\det(L)| \, d\lambda_d(u).
\]

\item[$(ii)$] For any oriented $C^1$ surface patch $S\subset\RE^d$, $\mathcal A(L(S)) = \cofactor(L)\,(\mathcal A(S))$, and in particular, at each smooth point $p \in S$, the oriented area element transforms as
\[
    \norm[L(S)](L p) \, d\lambda_{d-1}(L p)
    ~ = ~ \cofactor(L) \bigl( \norm[S](p) \, d\lambda_{d-1}(p) \bigr).
\]
\end{enumerate}
\end{lemma}

A transformation $\Phi$ on $\RE^d$ is called projective if it is of the form
\[
    \Phi(x) ~ = ~ \frac{M x + b}{\inner{c}{x} + \gamma},
\]
defined for those $x \in \RE^d$ satisfying $\inner{c}{x} + \gamma \neq 0$, where $M: \RE^d \to \RE^d$ is linear, $b,c \in \RE^d$, and $\gamma \in \RE$. Associated with $\Phi$ is the linear transformation $\widetilde{\Phi}: \RE^{d+1} \to \RE^{d+1}$ given by 
\[
    \widetilde{\Phi}(x,t) ~ := ~ (M x + t b, \, \inner{c}{x} + \gamma t). 
\]
We say that $\Phi$ is invertible, or nonsingular, if $\widetilde{\Phi}$ is nonsingular.

We call $\Phi$ \emph{admissible} for a convex body $K$ if 
\[
    \inner{c}{x} + \gamma \neq 0 \qquad \text{for all } x \in K. 
\]
Equivalently, $\Phi$ is well defined on $K$, that is, no point of $K$ is mapped to infinity. Since $K$ is connected, this is equivalent to requiring 
\[
    \inner{c}{x} + \gamma > 0 \qquad \text{for all } x \in K, 
\]
possibly after replacing $(M,b,c,\gamma)$ by $-(M,b,c,\gamma)$. (Note that invertibility and admissibility are logically independent in the sense that invertibility is a global projective condition, whereas admissibility is a domain condition relative to $K$.)

The following lemma describes how polarity behaves under admissible projective transformations, and in particular under translations. A proof of~(i) appears in~\cite{McS71}, and~(ii) follows by specializing to translations (see~\cite[Lemma~2.1]{FVW23}).

\begin{lemma}\label{lem:trans-polar}
Let $K$ be a convex body with the origin $O$ in its interior. Then:
\begin{enumerate}
\item[$(i)$] For any invertible projective transformation $\Phi$ admissible for $K$ with $O \in \interior(\Phi K)$, there exists an invertible projective transformation $\Phi'$ admissible for $\polar{K}$ such that $\polar{(\Phi K)} = \Phi'(\polar{K})$.

\item[$(ii)$] For $x \in \interior(K)$, let $P_x$ be the projective transformation defined on the open halfspace
\[
    \Dom(P_x) ~ := ~ \{y \in \RE^d \ST \inner{x}{y} < 1\}.
\]
by $P_x(y) := y/(1-\inner{x}{y})$. Then $P_x$ is admissible for $\polar{K}$ and $\polar{(K-x)} = P_x(\polar{K})$ (see Figure~\ref{f:polar-projective}).
\end{enumerate}
\end{lemma}

\begin{figure}[htbp]
  \centerline{\includegraphics[scale=0.40]{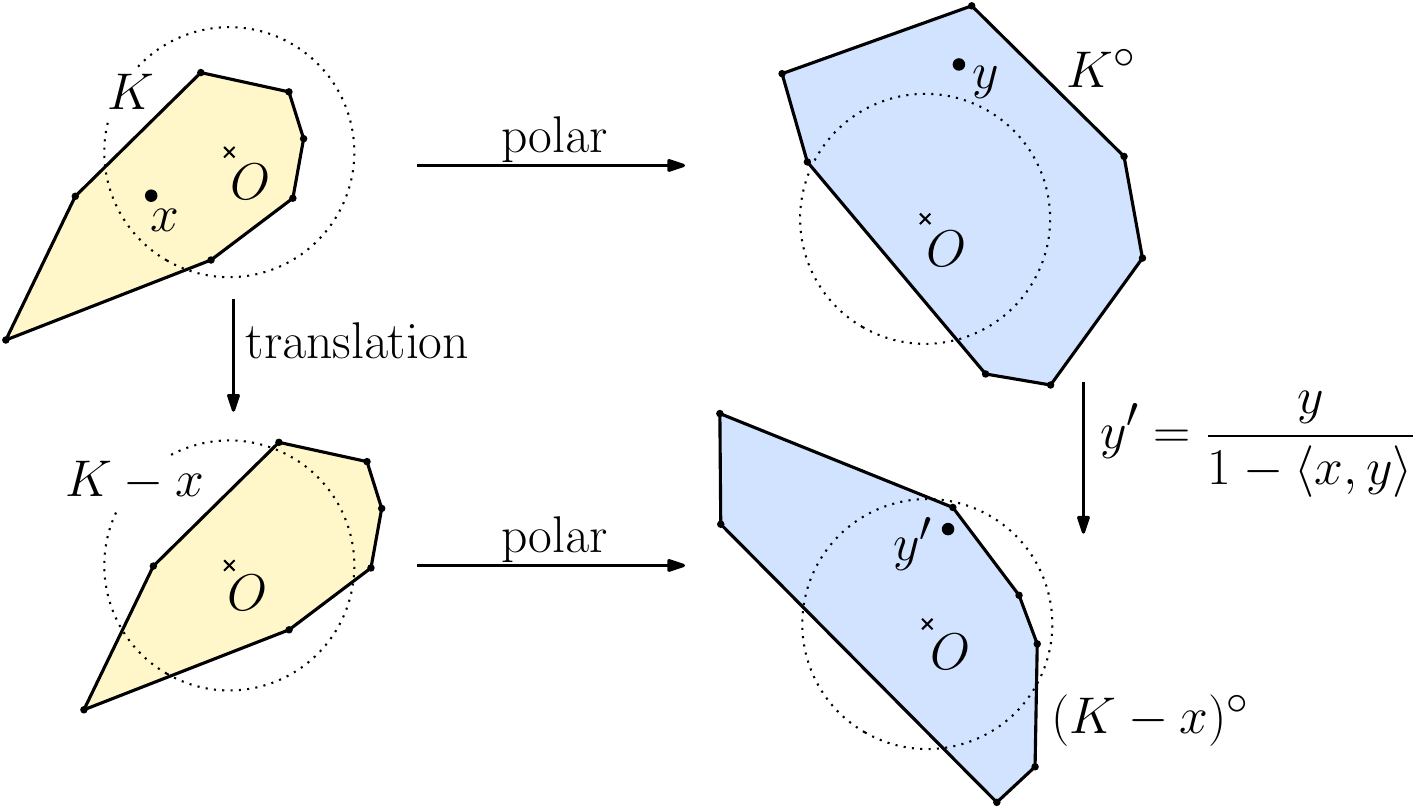}}
  \caption{Effect of translation on the polar.}
  \label{f:polar-projective}
\end{figure}

To evaluate the Funk Holmes--Thompson volume and area densities at
$x \in \interior(K)$, we need to understand how the polar body $\polar{(K-x)}$ depends on $x$. By Lemma~\ref{lem:trans-polar}(ii), this body is the image of $\polar{K}$ under the projective transformation $P_x$.

For $y\in\Dom(P_x)$, let $J_{x,y}$ denote the derivative of $P_x$ at $y$. A direct calculation gives
\[
    J_{x,y}(h)
        ~ = ~ \frac{(1-\inner{x}{y}) h + \inner{x}{h} y}{(1 - \inner{x}{y})^2},
\]
or equivalently,
\[
    J_{x,y}
        ~ = ~ (1-\inner{x}{y})^{-2} \bigl( (1-\inner{x}{y}) I + y \otimes x \bigr),
\]
where $(y \otimes x)(h) := \inner{x}{h} y$.

\begin{lemma}\label{lem:jacobian}
For any $x, y \in \RE^d$ such that $\inner{x}{y} < 1$:
\begin{enumerate}
\item[$(i)$] $J_{y,x} = J_{x,y}^T$.
\item[$(ii)$] $\det(J_{y,x}) = \det(J_{x,y})$.
\end{enumerate}
\end{lemma}

\begin{proof}
Using the operator formula above and the identity $(y \otimes x)^T = x \otimes y$, we obtain
\[
    J_{x,y}^T
        ~ = ~ (1-\inner{x}{y})^{-2} \bigl( (1 - \inner{x}{y}) I + x \otimes y \bigr)
        ~ = ~ J_{y,x},
\]
which proves~(i). Part~(ii) follows by taking determinants and using
$\det(L^T) = \det(L)$.
\end{proof}

\bigskip

We are now ready to prove the duality lemma.

\funkHTDuality*

\begin{proof}
To prove part~(i), fix $x \in G$. Since $x \in \interior(K)$, we have $1 - \inner{x}{y} > 0$ for all $y \in \polar{K}$. Thus $\polar{K} \subset \Dom(P_x)$, and since $\polar{K}$ is compact, $P_x$ is a $C^1$ diffeomorphism from an open neighborhood of $\polar{K}$ onto its image. By Lemma~\ref{lem:trans-polar}(ii),
\[
    \polar{(K-x)} ~ = ~ P_x(\polar{K}).
\]
Therefore, the change-of-variables formula yields
\[
    \lambda_d \bigl( \polar{(K-x)} \bigr)
        ~ = ~ \int_{\polar{K}} |\det(J_{x,y})| \, d\lambda_d(y).
\]
By the definition of the Funk Holmes--Thompson volume,
\[
    \volFunk[K](G)
        ~ = ~ \frac{1}{\omega_d} \int_G \left( \int_{\polar{K}} |\det(J_{x,y})| \, d\lambda_d(y) \right) d\lambda_d(x).
\]

Since $G\subset\interior(K)$, one has $\polar{K} \subset \interior(\polar{G})$. Applying the same argument with $G$ and $K$ interchanged, we obtain
\[
    \volFunk[\polar{G}](\polar{K})
        ~ = ~ \frac{1}{\omega_d} \int_{\polar{K}} \left( \int_G |\det(J_{y,x})| \, d\lambda_d(x) \right) d\lambda_d(y).
\]
By Lemma~\ref{lem:jacobian}(ii), $|\det(J_{y,x})| = |\det(J_{x,y})|$. Since the integrand is nonnegative, Tonelli's theorem applies~\cite[Theorem~8.8(a)]{Rud87}, and the two expressions coincide.

To prove part~(ii), fix a smooth point $x \in \bd G$. Since $x \in \interior(K)$, the map $P_x$ is a $C^1$ diffeomorphism from an open neighborhood of $\polar{K}$ onto its image. In particular, for every $y \in \bd \polar{K}$, we have $DP_x(y)=J_{x,y}$. At each smooth point $y \in \bd \polar{K}$, applying Lemma~\ref{lem:volume-area-transform}(ii) with $L = DP_x(y) = J_{x,y}$, we obtain
\[
    \norm[P_x(\polar{K})](P_x(y)) \, d\lambda_{d-1}(P_x(y))
        ~ = ~ \cofactor(J_{x,y}) \bigl( \norm[\polar{K}](y) \, d\lambda_{d-1}(y) \bigr).
\]

By Lemma~\ref{lem:trans-polar}(ii), $P_x(\polar{K}) = \polar{(K - x)}$. Hence, by Cauchy's projection formula applied to the convex body $P_x(\polar{K})$ in the direction $\norm[G](x)$, the $\lambda_{d-1}$-measure of the orthogonal projection of $P_x(\polar{K})$ onto the tangent hyperplane to $\bd G$ at $x$ is
\[
    \frac{1}{2} \int_{\bd\polar{K}} \left| \inner{\cofactor(J_{x,y}) \, \norm[\polar{K}](y)}{\norm[G](x)} \right|
    d\lambda_{d-1}(y).
\]
Using the fact that $\cofactor(J_{x,y}) = \det(J_{x,y})\,J_{x,y}^{-T}$, this becomes
\[
    \frac{1}{2} \int_{\bd\polar{K}} |\det(J_{x,y})| \cdot \bigl| \inner{J_{x,y}^{-T}\,\norm[\polar{K}](y)}{\, \norm[G](x)} \bigr| \, d\lambda_{d-1}(y).
\]

Substituting this expression into the definition of the Holmes--Thompson area and integrating over $x\in\bd G$, we find that $\areaFunk[K](G)$ is equal to
\begin{equation} \label{eq:areaFunk1}
    \areaFunk[K](G) ~ = ~
    \frac{1}{2\, \omega_{d-1}} \int_{\bd G} \left( \int_{\bd\polar{K}} \!\! |\det(J_{x,y})| \cdot         \bigl| \inner{J_{x,y}^{-T} \norm[\polar{K}](y)}{\, \norm[G](x)} \bigr| \, d\lambda_{d-1}(y) \right)    d\lambda_{d-1}(x).
\end{equation}
Similarly, interchanging the roles of the bodies shows that
$\areaFunk[\polar{G}](\polar{K})$ is equal to
\begin{equation} \label{eq:areaFunk2}
    \areaFunk[\polar{G}](\polar{K}) ~ = ~
    \frac{1}{2\, \omega_{d-1}} \int_{\bd\polar{K}} \!\!\left( \int_{\bd G} |\det(J_{y,x})| \cdot         \bigl| \inner{J_{y,x}^{-T}\norm[G](x)}{\norm[\polar{K}](y)} \bigr| \,d\lambda_{d-1}(x) \right)     d\lambda_{d-1}(y).
\end{equation}
By Lemma~\ref{lem:jacobian}, we have $J_{y,x} = J_{x,y}^T$ and $|\det(J_{y,x})| = |\det(J_{x,y})|$. Moreover,
\begin{align*}
    \inner{J_{y,x}^{-T}\norm[G](x)}{\, \norm[\polar{K}](y)}
    & ~ = ~ \inner{(J_{x,y}^T)^{-T}\norm[G](x)}{\, \norm[\polar{K}](y)} \\
    & ~ = ~ \inner{J_{x,y}^{-1}\norm[G](x)}{\, \norm[\polar{K}](y)} \\
    & ~ = ~ \inner{\norm[G](x)}{\, J_{x,y}^{-T} \norm[\polar{K}](y)}.
\end{align*}
Thus the integrands in~\eqref{eq:areaFunk1} and~\eqref{eq:areaFunk2} are equal. Since they are nonnegative, Tonelli's theorem applies, and the two double integrals are equal. This proves the claim.
\end{proof}

\section{Proof of the Hilbert Boundary-Transfer Property} \label{app:hilb-exp-surf-i}

In this section, we present a proof of Lemma~\ref{lem:hilb-exp-bdry}, which states that for every point on the boundary of a convex body $G$ in the Hilbert geometry, there is a point on the boundary of its $\alpha$-expanded body that is at distance $\alpha$.

\lemHilbExpBdry*

Throughout this section, $G$, $G_+$, $K$, and $\alpha$ are as given in the lemma's statement. We will use $\dist(\cdot,\cdot)$ as a shorthand for the Hilbert distance with respect to $K$. We begin with the following inequality, due to Busemann (see \cite[Eq.~(18.7)]{Bus55}).

\begin{lemma}[Busemann's pencil inequality] \label{lem:busemann-18-7}
Given a convex body $K$ and any points $p,q \in \interior(K)$, let $a$ and $b$ denote the points where the line through $p$ and $q$ meets $\bd K$. Let $h_a$ and $h_b$ denote any supporting hyperplanes of $K$ at $a$ and $b$, respectively, and let $\mathcal{P}$ be the pencil of hyperplanes passing through $h_a \cap h_b$ in projective space (see Figure~\ref{f:comp-chord}(a)). If $h_p, h_q \in \mathcal{P}$ are the hyperplanes with $p \in h_p$ and $q \in h_q$, then for any $u \in h_p \cap \interior(K)$ and any $v \in h_q \cap \interior(K)$, $\dist(p,q) \leq \dist(u,v)$.
\end{lemma}

\begin{figure}[htbp]
  \centerline{\includegraphics[scale=0.40,page=1]{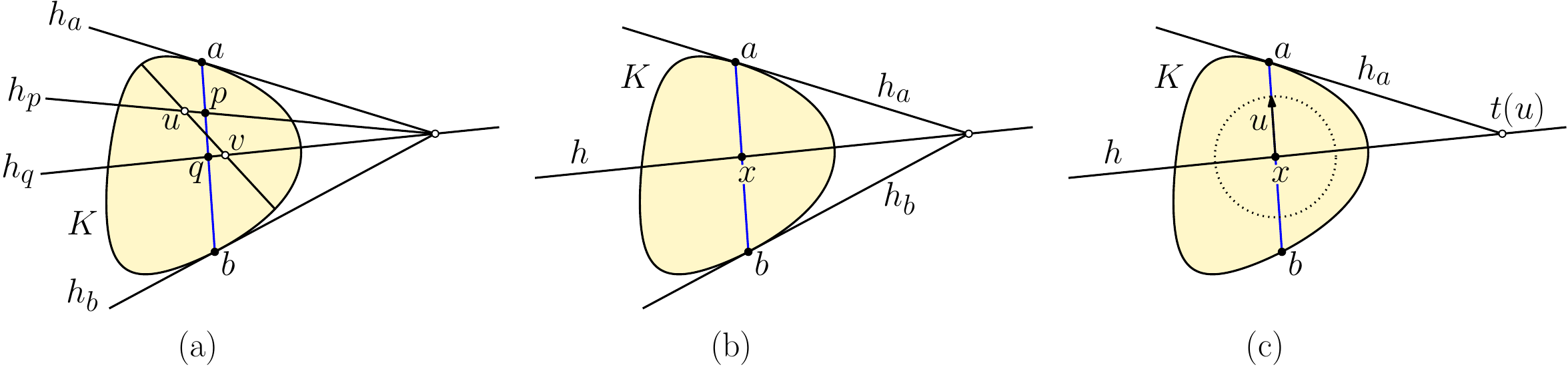}}
  \caption{(a) Busemann's pencil inequality, (b) complementary chords, and (c) Lemma~\ref{lem:comp-chord}.}
  \label{f:comp-chord}
\end{figure}

It will be useful to introduce an adaptation of the Euclidean notion of perpendicularity to Hilbert geometry. Fix $x \in \interior(K)$,  and let $h$ be any hyperplane through $x$. A chord $\overline{a b}$ of $K$ passing through $x$ is said to be \emph{complementary} to $h$ if there exist supporting hyperplanes $h_a$ and $h_b$ at $a$ and $b$, respectively, such that $h_a \cap h_b \subset h$ (see Figure~\ref{f:comp-chord}(b)). (Note that $h_a \cap h_b$ is understood to be a $(d-2)$-dimensional flat in projective space, and it may lie at infinity.) Observe that complementarity is invariant under projective transformations.

\begin{lemma}[Complementary chord existence] \label{lem:comp-chord}
Given a convex body $K$ in $\RE^d$, any point $x \in \interior(K)$, and any hyperplane $h$ passing through $x$, there exists a chord through $x$ that is complementary to $h$.
\end{lemma}

\begin{proof}
We first treat the case where $K$ is smooth. We will employ the Borsuk--Ulam theorem~\cite{Bor33}, which states that if $f: S^n \to \RE^n$ is continuous, then there exists $u \in S^n$ such that $f(-u) = f(u)$. 

Here $n = d-1$. For $u \in S^{d-1}$, let $a = a(u)$ be the point where the ray from $x$ in direction $u$ meets $\bd K$ (see Figure~\ref{f:comp-chord}(c)). Let $h_a$ be the unique supporting hyperplane of $K$ at $a$, and let $t(u)$ denote the $(d-2)$-flat $h\cap h_a$ in projective space, which may lie at infinity.

Identify $h$ with $\RE^{d-1}$ with the origin at $x$. If $t(u)$ lies in the affine chart, then it is an affine
hyperplane in $\RE^{d-1}$ not passing through the origin. In this case there exists a unique vector $t(u)^* \in \RE^{d-1} \setminus \{O\}$ such that
\[
    t(u)
        ~ = ~ \{z \in \RE^{d-1} \ST \inner{t(u)^*}{z} = 1\}. 
\]
If $t(u)$ lies at infinity (that is, if $h_a$ and $h$ are parallel), set $t(u)^* := 0$. Define $f: S^{d-1} \to \RE^{d-1}$ by $f(u) := t(u)^*$. In the smooth case, the map $u \mapsto h_a$ varies continuously, and hence so does $f$. 

By Borsuk--Ulam there exists $u \in S^{d-1}$ such that $f(-u) = f(u)$. Let $b := a(-u)$, and let $h_b$ be the unique
supporting hyperplane of $K$ at $b$. The equality $f(-u) = f(u)$ implies that $t(-u) = t(u)$, that is,
\[
    h \cap h_a ~ = ~ h \cap h_b 
\]
as $(d-2)$-flats in projective space. Equivalently, the hyperplanes $h$, $h_a$, and $h_b$ are concurrent in projective space, and hence the chord $\overline{a b}$ through $x$ is complementary to $h$, as desired.

The general (non-smooth) case is handled by a limit argument. Consider smooth convex bodies $K_{\eps} \supseteq K$ with $\lim_{\eps \to 0} K_{\eps} = K$. Applying the smooth case to each $K_{\eps}$ yields a chord $\overline{a_{\eps} b_{\eps}}$ through $x$ and supporting hyperplanes $h_{a_{\eps}}, h_{b_{\eps}}$ of $K_{\eps}$ at $a_{\eps},b_{\eps}$, respectively, such that $(h_{a_{\eps}} \cap h_{b_{\eps}}) \subset h$.

By compactness, after passing to a sequence $\eps_n \to 0$, we may assume that 
\[
    a_{\eps_n} \to a, \qquad 
    b_{\eps_n} \to b, \qquad 
    h_{a_{\eps_n}} \to h_a, \qquad 
    h_{b_{\eps_n}} \to h_b, 
\]
where $a, b \in \bd K$ and $h_a, h_b$ are supporting hyperplanes of $K$ at $a$ and $b$, respectively. Passing to the limit yields $(h_a \cap h_b) \subseteq h$, and therefore $\overline{a b}$ is complementary to $h$, as desired.
\end{proof}

\bigskip

We are now in a position to present the proof.

\begin{proof} (of Lemma~\ref{lem:hilb-exp-bdry})
Fix $x \in \bd G$, and let $h_x$ be any supporting hyperplane of $G$ at $x$. Apply Lemma~\ref{lem:comp-chord} with $h:=h_x$ to obtain a chord $\overline{ab}$ of $K$ through $x$ and supporting hyperplanes $h_a,h_b$ of $K$ at $a,b$ such that
\[
    t ~ := ~ (h_a \cap h_b )\subseteq h_x 
\]
in projective space (see Figure~\ref{f:hilb-exp-bdry}(a)).

\begin{figure}[htbp]
  \centerline{\includegraphics[scale=0.40,page=2]{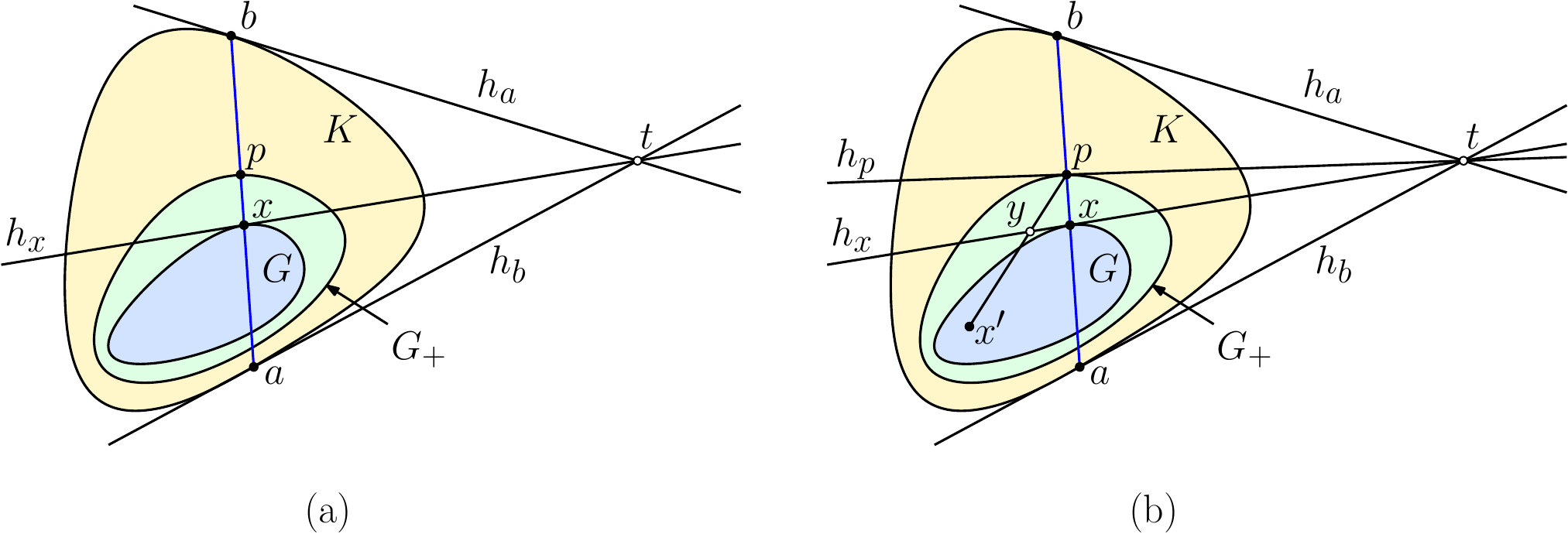}}
  \caption{Proof of Lemma~\ref{lem:hilb-exp-bdry}.}
  \label{f:hilb-exp-bdry}
\end{figure}

Label $a, b$ so that $b$ lies on the side of $h_x$ disjoint from $\interior(G)$. Choose $p\in (x,b)$ such that $\dist(x,p) = \alpha$. Such a point exists by continuity, since $\dist(x,q) \to +\infty$ as $q \to b$. Because $x\in G$, we have $\dist(p,G) \leq \dist(p,x) = \alpha$, and hence $p \in G_+$.

Now, take any $x' \in G$ and set $y := \overline{p x'} \cap h_x$. Then $p$, $y$, and $x'$ are collinear with $y$ between $p$ and $x'$, so $\dist(p, x') \geq \dist(p, y)$ (see Figure~\ref{f:hilb-exp-bdry}(b)). Let $h_p$ be the hyperplane through $p$ and $t$. Applying Lemma~\ref{lem:busemann-18-7} to the pair $(p, x)$, with the pencil determined by $h_a \cap h_b$, for $u := p \in h_p$ and $v := y\in h_x$, we have 
\[
    \dist(p, y)
        ~ \geq ~ \dist(p, x)
        ~ =    ~ \alpha,
\]
and therefore $\dist(p, x') \geq \alpha$, for all $x' \in G$. Taking the infimum over $x' \in G$ yields $\dist(p, G) \geq \alpha$. Therefore $\dist(p, G) = \alpha$, so $p \in \bd G_+$. Setting $y := p$, we obtain $\dist(x, y) = \alpha$. The final inclusion follows immediately, since for every $x \in \bd G$, $\dist(x, \bd G_+) \leq \alpha$.
\end{proof}

\section{Auxiliary Proofs for Boundary Duality by Translates} \label{app:asym-boundary-mink}

In this appendix, we prove the two auxiliary lemmas from Section~\ref{s:boundary-cover}. We also recall a Cauchy formula for Holmes--Thompson surface area in a Minkowski space, due to Holmes and Thompson~\cite{HoT79}, which will be used in the proof of Lemma~\ref{lem:sym-union-HT}.

\begin{lemma} \label{lem:mink-cauchy-HT}
Given convex bodies $C$ and $D$ in $\RE^d$ where $D$ is centrally symmetric,
\[
    \areaMink[D](C)
        ~ = ~ \frac{1}{\omega_{d-1}} \int_{\bd \polar{D}} \lambda_{d-1} \bigl( \orthproj{C}{\norm[D](x)^\perp} \bigr) \, d\lambda_{d-1}(x).
\]
\end{lemma}

\lemSymUnionHT*

\begin{proof}
Given a convex body $C$, its \emph{difference body}, denoted $\Delta(C)$, is the centrally symmetric set $C + (-C)$. Because $O \in C$,
\[
    C_{\cup}
        ~ = ~ \conv(C \cup -C)
        ~ \subseteq ~ \Delta(C).
\]
Hence, Rogers--Shephard~\cite{RoS57} gives 
\[
    \volMink[D](C_{\cup})
        ~ =    ~ \frac{\lambda_d(\polar D)}{\omega_d}\,\lambda_d(C_{\cup})
        ~ \leq ~ \frac{\lambda_d(\polar D)}{\omega_d}\,\lambda_d(\Delta(C))
        ~ \leq ~ \binom{2 d}{d} \volMink[D](C),
\]
proving~(i).

For part~(ii), let $x \in \bd \polar D$ be a smooth point, and set
\[
    P_x
        ~ :=      ~ \orthproj{C}{\norm[D](x)^\perp}
        ~ \subset ~ \norm[D](x)^\perp.
\]
Then
\[
    \orthproj{C_{\cup}}{\norm[D](x)^\perp}
        ~ =         ~ \conv(P_x \cup -P_x)
        ~ \subseteq ~ \Delta(P_x).
\]
Applying Rogers--Shephard in the $(d-1)$-dimensional space $\norm[D](x)^\perp$, we obtain
\[
    \lambda_{d-1}\bigl(\orthproj{C_{\cup}}{\norm[D](x)^\perp}\bigr)
        ~ \leq ~ \binom{2d-2}{d-1}\, \lambda_{d-1}(P_x)
        ~ =    ~ \binom{2d-2}{d-1}\, \lambda_{d-1} \bigl( \orthproj{C}{\norm[D](x)^\perp} \bigr).
\]
Integrating over $\bd \polar D$ and using Lemma~\ref{lem:mink-cauchy-HT}, we get 
\[
    \areaMink[D](C_{\cup})
        ~ \leq ~ \binom{2d-2}{d-1}\, \areaMink[D](C).
\]
This proves~(ii).
\end{proof}

\lemCoreCover*

\begin{proof}
It suffices to show that each translate of $D$ can be covered by $\CLEQ 1$ translates of $D_{\cap}$. Let $T \subseteq D$ be maximal such that the translates $\{ t + \frac{1}{2} D_{\cap} \ST t \in T\}$ have pairwise disjoint interiors. Then clearly the translates $\{ t + D_{\cap} \ST t \in T \}$, cover $D$. Moreover, each half-sized translate is contained in 
\[
    D + \tfrac{1}{2} D_{\cap}
        ~ \subseteq ~ \tfrac{3}{2} D.
\]
Therefore,
\[
    |T|\, \lambda_d \left( \tfrac{1}{2} D_{\cap} \right) 
        ~ \leq ~ \lambda_d \left( \tfrac{3}{2} D \right), 
\]
and hence
\[
    |T|
        ~ \leq ~ 3^d\, \frac{\lambda_d(D)}{\lambda_d(D_{\cap})}.
\]

By a result of Milman--Pajor~\cite[Corollary~3 and the subsequent Remark~4]{MiP00}, if either the centroid of $D$ or the Santal\'o point of $D$ coincides with the origin, then $\lambda_d(D_{\cap}) \CGEQ \lambda_d(D)$. Thus, $|T| \CLEQ 1$. Refining each member of an optimal $D$-cover of $U$ by such a $D_{\cap}$-cover completes the proof. 
\end{proof}

\end{document}